\documentclass[amsfonts]{article}
\usepackage{amsmath}
\usepackage{amsfonts}
\usepackage{amsthm}
\usepackage{amssymb}
\usepackage{enumerate}
\theoremstyle{plain}
\newtheorem{theo}{Theorem}[section]
\newtheorem{prop}[theo]{Proposition}
\newtheorem{definition}[theo]{Definition}
\theoremstyle{remark}
\newtheorem{rem}[theo]{Remark}
\newcommand{\HH}{{\mathbb H}}
\newcommand{\RR}{{\mathbb{R}}}
\newcommand{\CC}{{\mathbb C}}
\newcommand{\NN}{{\mathbb N}}
\newcommand{\II}{{\mathbb I}}
\newcommand{\ds}{\displaystyle}
\newcommand{\der}{\partial}
\newcommand{\ep}{\varepsilon}
\newcommand{\ov}{\overline}

\newcommand{\om} {\Omega}

\newcommand{\nabhat}{{\widehat{\nabla}}}

\def\qed{\hfill$\square$\vspace{0.5cm}}    

\begin{document}

\title{Uniqueness and Lipschitz stability for the identification of Lam\'e parameters from boundary measurements }
\author{Elena Beretta\thanks{ Dipartimento di Matematica ``G. Castelnuovo''
Universit\`a
 di Roma ``La Sapienza''
(beretta@mat.uniroma1.it).} \and Elisa Francini\thanks{Dipartimento di Matematica e Informatica ``U. Dini'',
Universit\`{a} di  Firenze (francini@math.unifi.it)}\and
Sergio Vessella\thanks{Dipartimento di Matematica e Informatica ``U. Dini'',
Universit\`{a} di  Firenze (sergio.vessella@dmd.unifi.it)} }
\date{}
\maketitle
\begin{abstract} In this paper we consider the problem of determining an unknown pair $\lambda$, $\mu$ of piecewise constant Lam\'{e} parameters inside a three dimensional body from the Dirichlet to Neumann map. We prove uniqueness and Lipschitz continuous dependence of $\lambda$ and  $\mu$ from the  Dirichlet to Neumann map.
\end{abstract}
\section{Introduction\label{sec0}}
A relevant inverse problem arising in nondestructive testing of materials is the one of determining, within an isotropic, linearly elastic three dimensional body  $\Omega$, the elastic properties of the body from measurements of traction
 and displacement taken on the exterior boundary of the domain $\Omega$.\\
This leads mathematically to the formulation of the following boundary value problem for the system of linearized elasticity
\begin{equation}\label{iep}
    \left\{\begin{array}{rcl}
             \mbox{div}(\CC  \nabhat u)& = & 0\mbox{ in }\om\subset\RR^3,\\
             u & = & \psi\mbox{ on }\der\om,
           \end{array}\right.
\end{equation}
where $\om$ is an open and bounded domain, $\nabhat u$ denotes the strain tensor $\nabhat u:=\frac{1}{2}\left(\nabla u+\left(\nabla u\right)^T\right)$, $\psi\in H^{1/2}(\der\om)$ is the boundary displacement field, and $\CC\in L^{\infty}(\om)$ denotes the isotropic elasticity tensor with Lam\'e coefficients $\lambda, \mu$:
\[
\CC= \lambda
 I_{3}\otimes I_{3}+2\mu \II_{sym}, \hbox{ a.e. in } \Omega
\]
where
$I_3$ is $3\times 3$ identity matrix and  ${\II}_{Sym}$ is the fourth order tensor such that ${\II}_{Sym}A=\hat A$,
The strong convexity condition is assumed
\[ \mu\geq \alpha_0>0,\quad 2\mu+3\lambda\geq \beta_0>0 \hbox{ a.e. in } \Omega.\]
Under the above assumptions problem (\ref{iep}) has a unique weak solution $u\in H^1(\om)$ and the Dirichlet-to-Neumann linear map (DN map), $\Lambda_{\CC}$, is well defined
\[
\Lambda_{\CC}:\psi\in  H^{1/2}(\om)\rightarrow (\CC\nabhat u)\nu|_{\der\om}\in H^{-1/2}(\der\om)
\]
where $\nu$ is the exterior unit normal to $\der\om$.\\
The inverse problem consists in determining $\CC$, i.e. $\lambda$ and $\mu$, from knowledge of  the DN map $\Lambda_{\CC}$. \\
This problem is closely related to the conductivity inverse problem arising in the modelling of EIT (Electrical Impedence Tomography). For the mathematical treatment of this problem we refer to the fundamental papers \cite{Al1},\cite{AP}, \cite{N} and \cite{SU}. \\
Unfortunately, the mathematical approach used to investigate the conductivity inverse problem fails partly when dealing in the elasticity framework. This is due to the fact that we have to deal with an elliptic system instead  that with a scalar equation and we have to recover two parameters $\lambda$ and $\mu$ instead of the sole conductivity parameter.
 As a consequence of these difficulties only partial results to this inverse problem are known and mainly concern the uniqueness issue.
More precisely, the study  of the problem was initiated in the 90's by Ikeata in \cite{Ik} who considered a linearized version of it. In two dimensions Akamatsu, Nakamura and Steinberg  in \cite{ANS} and  for dimension $n\geq 3$  Nakamura and Uhlmann in \cite{NU95}  showed that one can determine uniquely and in a stable way $C^{\infty}(\overline{\om})$ Lam\'e parameters $\lambda$ and $\mu$ and their derivatives on the boundary of a smooth domain $\om$ from the DN map.
Local uniqueness has been proved in dimensions two by Nakamura and Uhlmann in \cite{NU1} for  $C^{\infty}(\overline{\om})$ Lam\'e parameters assuming that they are both close to positive constants.
In three dimensions and higher  Nakamura and Uhlmann in \cite{NU2} and Eskin and Ralston in \cite{ER} proved local uniqueness for smooth Lam\'e parameters whenever $\mu$ is close to a constant.\\
To our knowledge no result concerning stability estimates is known. Based on the results obtained by Alessandrini in \cite{Al1}  who proved, {\it logarithmic} stability estimates for the conductivity inverse problem in the case of smooth conductivities and the example of Mandache in \cite{M} who proved the optimality of this estimate also for the inverse elasticity problem in the case of smooth Lam\'e parameters  {\it logarithmic}  stability estimates or even worse ones are expected. \\
These considerations lead in recent years to look for different a priori assumptions on the unknown parameters which take into account the applied context from which the problem arises and give rise to better stability estimates (\cite{V}, \cite{ABV}, \cite{BF}, \cite{BFV}, \cite{ABF},\cite{BdHQ}). An attempt in this direction has been done by Alessandrini and Vessella in \cite{AV} with for the conductivity inverse problem for unknown conductivities depending only on a finite number of parameters. In \cite{AV} they proved {\it Lipschitz} continuous dependence  from the  DN map for conductivities that are constant on known subdomains, assuming ellipticity and $C^{1,\alpha}$ regularity at the interfaces joining contiguous domains.\\
In this paper we propose to consider $L^{\infty}(\om)$ elasticity tensors of the form
\[
\CC(x)=\sum_{j=1}^N(\lambda_j I_{3}\otimes I_{3}+2\mu_j\II_{Sym})\chi_{D_i}(x)
\]
where the $D_j$' s, $j=1,\cdots,N$, are known disjoint  Lipschitz domains representing a partition of $\om$ and $\lambda_j,  \mu_j$,  $j=1,\cdots,N$, are unknown constants.\\
We will prove that if $\CC^1$ and $\CC^2$ are of this form, assuming that the boundaries of $D_j$'s and of $\om$ contain flat portions, we have
\[
\|\CC^{1}-\CC^{2}\|_{\infty}\leq C\|\Lambda_{\CC^{1}}-\Lambda_{\CC^{2}}\|_{\mathcal{L}\left(H^{1/2}(\partial\Omega),H^{-1/2}(\partial\Omega)\right)}.
\]
where the stability constant $C$ appearing in the estimate depends on various parameters like $\alpha_0,\beta_0$, on the regularity bounds on $\om$ and on the $D_j$' s and on their number $N$. In particular,  the constructive character of the proof allows us to establish an estimate from above of the constant $C$.\\
We want now to emphasize that several significant examples fit in our analysis. Polyhedral partitions of $\om$, appearing in any finite-element scheme used for effective reconstruction of the Lam\'e parameters  (see for example \cite{BJK}; and a  layered configuration of the sets $D_j$'s arising in the study of composite laminates \cite{Mi} and in geophysical prospection, \cite{BC}. \\
Our approach is based on the use of the following key ingredients: existence of singular solutions and study of their behaviour close to the flat discontinuity interfaces of the $D_j$'s, regularity estimates and quantitative estimates of unique continuation of solutions  to system (\ref{iep}).\\ As already pointed out in \cite{BF} a relevant difference with respect to the scalar case treated in \cite{AV}  is the issue of existence of singular solutions. In fact, in the case of strongly elliptic systems with $L^{\infty}$ coefficients in dimension $n\geq 3$,  existence of the fundamental solution and of the Green's function cannot be inferred without additional assumptions.  In \cite{HK} Hofmann and Kim prove their existence under the additional information that weak solutions  of the system satisfy De Giorgi-Nash type local H\"{o}lder estimates. It is clear that in the case of a polyhedral partition of $\om$ solutions might not enjoy H\"{o}lder regularity at edges. On the other hand, in order to obtain our result, it is enough to construct singular solutions and analyze their behavior in a Lipschitz subset $\tilde{\mathcal K}$ at a given positive distance from edges.
Nevertheless, while for the scalar (even for the complex valued  treated in \cite{BF}) conductivity equation is fairly easy to get a solution of the equation close to a flat interface by using the fundamental solution for the Laplace equation and suitable reflection arguments, this seems not to be possible for the  Lam\'e system.
In order to construct singular solutions, we make use of special fundamental solutions constructed by Rongved in \cite{R} for  isotropic  biphase laminates.
Furthermore, looking at solutions of the elasticity system in $\tilde{\mathcal K}$ we can use the results of \cite{CKVC} and of \cite{LN}  deriving regularity estimates for the solutions which allow us to obtain H\"{o}lder estimates of unique continuation in $\tilde{\mathcal K}$.\\
We would like to point out that our proof is based on the use of solutions having boundary displacement fields supported in the flat portion $\Sigma_1$ of $\der\om$. Hence, our stability result also holds replacing the full DN map with the {\it local} DN map that we will define in Section 2.
Moreover, we expect to derive similar stability estimates also the case of domains $D_j$'s  with $C^{1,\alpha}$ portions of interfaces and this analysis will be object of a forthcoming publication. \\
Finally, we would like to emphasize that {\it Lipschitz} stability estimates have become crucial also for the effective reconstruction of the unknown coefficients. In fact, recently, in \cite{dHQS1} and \cite{dHQS2}, de Hoop, Qiu and Scherzer have shown that  Lipschitz type stability estimates imply local convergence of iterative reconstruction algorithms and the radius of convergence of the iterates depends on the stability constant and hence an explicit determination of the dependence of such constant from the a-priori parameters, in particular from the partition number $N$, is crucial.

\bigskip

The plan of the paper is the following: section \ref{sec2} contains the description of the main result. In paragraph \ref{subsec2.1}  we introduce the notation and  main definitions. In paragraph \ref{subsec2.3}  we state the main a priori assumptions and the main result. We also reformulate the inverse elasticity problem in terms of the nonlinear forward map $F$ acting on a finite-dimensional subset of $\RR^{2N}$ and recall a general result (Proposition \ref{propBV}) that will let us show that $F$ has a Lipschitz continuous inverse.

Section \ref{prelres} contains some auxiliary results. In particular, in paragraph \ref{sec4} we collect some properties concerning the fundamental solution in biphase elastic isotropic materials introduced by Rongved in \cite{R} and we prove existence of singular solutions for our Lam\'{e} system. In paragraph \ref{reg}  we  state some regularity results and estimates of unique continuation concerning  solutions of piecewise constant Lam\'{e} systems.

In section \ref{sec5} we give the proof of our main result by verifying that the forward map $F$ corresponding to our inverse problem, satisfies all the assumptions of  Proposition \ref{propBV} thus concluding the proof.

Finally, in the Appendix we recall some known results concerning solutions to Lam\'{e} systems with constant coefficients and the proof of the estimate of unique continuation stated in section \ref{prelres}.
\subsubsection*{Aknowledgements}
We want to thank Antonino Morassi for pointing out to us the paper by Rongved on biphase fundamental solution and for stimulating discussions.
This work has been supported by MIUR within the project PRIN 20089PWTPS003.
\section{Main result}\label{sec2}
\subsection{Notation and main definitions}\label{subsec2.1}
For every $x\in\RR^3$ we set  $x=(x^\prime,x_3)$ where $x^\prime\in\RR^{2}$ and $x_3\in \RR$.
For every $x\in \RR^3$, $r$ and $L$ positive real numbers we will denote by $B_r(x)$, $B_r^\prime(x^\prime)$ and $Q_{r,L}(x)$  the open ball in $\RR^3$ centered
at $x$ of radius $r$, the open ball in $\RR^2$ centered at $x^\prime$ of radius $r$ and the cylinder $B_r^\prime(x^\prime)\times (x_3-Lr,x_3+Lr)$, respectively; in the sequel $B_r(0)$, $B_r^\prime(0)$ and $Q_{r,L}(0)$ will be denoted by $B_r$, $B^\prime_r$ and $Q_{r,L}$, respectively.
 We will also denote by $\RR^3_+=\{(x^\prime,x_3)\in\RR^3\,:\,x_3>0\}$,
$\RR^3_-=\{(x^\prime,x_3)\in\RR^3\,:\,x_3<0\}$, $B^+_r=B_r\cap \RR^3_+$, and $B^-_r=B_r\cap \RR^3_-$.

For any subset $D$ of $\RR^3$ and any $h>0$, we denote by
\[(D)_{h}=\{x \in D| \mbox{ dist}(x,\RR^3\setminus D)>h\}.\]

\begin{definition}\label{lipboundary}
Let $\om$ be a bounded domain in $\RR^3$. We shall say that a portion $\Sigma\subset\der\om$
is of Lipschitz class with constants $r_0>0, L\geq 1$ if for any point $P\in \Sigma$, there exists a rigid transformation of coordinates under which $P=0$ and
\[
\om\cap Q_{r_0,L}=\{(x^\prime,x_3) \in Q_{r_0,L}|x_3>\psi(x^\prime)\}.
\]
where  $\psi$ is a Lipschitz continuous function in $B_{r_0}^\prime$ such that
\[
\psi(0)=0\mbox{ and }\,
\|\psi\|_{C^{0,1}(B_{r_0}^\prime)}\leq Lr_0.
\]
We say that $\om$ is of Lipschitz class with constants $r_0$ and $L$ if $\der\om$ is of Lipschitz class with the same constants.
\end{definition}

\begin{rem}
We use the convention of normalizing all norms in such a way that all their terms are dimensionally homogeneous. For example:
\[
\|\psi\|_{C^{0,1}(B_{r_0}^\prime)}=\|\psi\|_{L^{\infty}(B_{r_0}^\prime)}+r_0\|\nabla\psi\|_{L^{\infty}(B_{r_0}^\prime)}.
\]
Similarly, denoting by $D^i u$ the vector which components are the derivatives of order $i$ of the function $u$,
\begin{equation*}
\|u\|_{L^2(\Omega)}=\left(r_0^{-3}\int_\Omega
u^2\right) ^{\frac{1}{2}},
%
 \quad
  \|u\|_{{C}^{k}(\Omega)} =\sum_{i=0}^{k}
  {r_{0}}^{i}\|D^{i} u\|_{{L}^{\infty}(\Omega)},
\end{equation*}
\begin{equation*}
\|u\|_{H^m(\Omega)}=r_0^{-3/2}\left(\sum_{i=0}^m
r_0^{2i}\int_\Omega|D^i u|^2\right)^{\frac{1}{2}},
\end{equation*}
and so on for boundary and trace norms such as
$\|\cdot\|_{H^{\frac{1}{2}}(\partial \Omega)}$, $\|\cdot\|_{H^{-\frac{1}{2}}(\partial \Omega)}$, where $\Omega$  is a bounded subset of $\RR^3$ whose boundary is smooth enough.
\end{rem}

We will also make use of the following notations for matrices and tensors:
for any $3\times 3$ matrices $A$ and $B$  we set  $A : B=\sum_{i,j=1}^3 A_{ij}B_{ij}$
 and $\hat A=\frac{1}{2}(A+A^T)$.
By $I_3$ we denote the $3\times 3$ identity matrix and by ${\II}_{Sym}$ we denote the fourth order tensor such that ${\II}_{Sym}A=\hat A$.

 \bigskip

In the whole paper we are going to consider isotropic elastic materials, hence the elasticity tensor $\CC$ is a fourth order tensor given by
\begin{equation}\label{isotr}
\CC(x)=\lambda(x)I_3\otimes I_3 +2\mu(x){\II}_{Sym}\quad\hbox{ for  a.e.  }x\mbox{ in } \Omega,
\end{equation}
where $\Omega$ is a bounded domain in $\RR^3$ of Lipschitz class, and the real valued functions $\lambda=\lambda(x)$ and $\mu=\mu(x)\in L^{\infty}(\om)$ are the Lam\'{e} moduli.
We will also use Poisson's ratio  $\nu(x)=\frac{\lambda(x)}{2(\lambda(x)+\mu(x))}$.

 \bigskip

An elasticity tensor $\CC$ is \textit{strongly convex} if there is a positive number $\xi_0$ such that,
for almost every $x$ in $\Omega$,
\begin{equation}\label{strconv}
    \CC(x) A: A\geq \xi_0|A|^2 \quad \mbox{ for every }3\times 3 \mbox{ symmetric matrix }A.
\end{equation}
In the isotropic case (\ref{isotr}), the strong convexity condition takes the form
\begin{equation}\label{strconvlame} \mu(x)\geq\alpha_0>0,\quad 2\mu(x)+3\lambda(x)\geq \beta_0 \quad\hbox{ for  a.e.  }x\mbox{ in } \Omega.\end{equation}
In these case, the Poisson's ratio has values in an compact subset of $(-1,\frac{1}{2})$.
More precisely we can estimate
\begin{equation}\label{boundPoisson}
-1+\frac{\beta_0\alpha_0}{4}\leq\nu(x)\leq \frac{1}{2}-\frac{\alpha^2_0}{4}
 \quad\hbox{ for  a.e.  }x\mbox{ in } \Omega.\end{equation}

In the sequel we will make use of the following norm in the linear space of isotropic tensors:
\[\|\CC\|_{\infty}=\max\left\{\|\lambda\|_{L^{\infty}(\om)},\|\mu\|_{L^{\infty}(\om)}\right\}.\]
This norm is equivalent to the usual $L^\infty$ norm for tensors in the space of isotropic tensors.

Our boundary measurements are represented by the Dirichlet to Neumann map.
As a matter of fact, since we will restrict our measurements to boundary data that have support on some subset of the boundary, we will make use of a local Dirichlet to Neumann map.

\begin{definition}[The Local Dirichlet to Neumann map]
\label{DNmap}
Let $\om$ be a bounded domain of Lipschitz class and let $\Sigma$ be an open portion of $\der\om$.
We denote by $H^{1/2}_{co}(\Sigma)$ the function space
\[
H^{1/2}_{co}(\Sigma):=\left\{\phi\in H^{1/2}(\der\om)\,:\,\mbox{supp }\phi\subset\Sigma\right\}
\]
and by $H^{-1/2}_{co}(\Sigma)$ the topological dual of $H^{1/2}_{co}(\Sigma)$. We denote by $<\cdot,\cdot>$ the dual pairing between $H^{1/2}_{co}(\Sigma)$ and
$H^{-1/2}_{co}(\Sigma)$ based on the $L^2(\Sigma)$ scalar product.
 Then, given $\psi\in H^{1/2}_{co}(\Sigma)$,  there exists a unique vector valued function $u\in H^1(\om)$ weak solution to the Dirichlet
 problem
\begin{equation}\label{1}
    \left\{\begin{array}{rcl}
             \mbox{div}(\CC  \nabhat u)& = & 0\mbox{ in }\om, \\
             u & = & \psi\mbox{ on }\der\om.
           \end{array}\right.
\end{equation}
We define the local Dirichlet to Neumann linear map $\Lambda_{\CC}^{\Sigma}$ as follows:
\[
\Lambda_{\CC}^{\Sigma}:\psi\in  H^{1/2}_{co}(\Sigma)\rightarrow (\CC\nabhat u)n|_{\Sigma}\in H^{-1/2}_{co}(\Sigma).
\]
\end{definition}

Note that for $\Sigma=\der\om$ we get the usual Dirichlet to Neumann map. For this reason we will set $\Lambda_{\CC}:=\Lambda_{\CC}^{\der\om}$.

 \bigskip

The map $\Lambda_{\CC}^{\Sigma}$ can be identified with the bilinear form on $H^{1/2}_{co}(\Sigma)\times H^{1/2}_{co}(\Sigma)$ by
\begin{equation}\label{DNbilineare}
\widetilde{\Lambda}_{\CC}^{\Sigma}(\psi, \phi):=<\Lambda_{\CC}^{\Sigma}\psi, \phi>=\int_{\om}\CC \nabhat u:\nabhat v
\end{equation}
for all $\psi,\phi \in H^{1/2}_{co}(\Sigma)$ and where $u$ solves (\ref{1})
and $v$ is any $H^1(\om)$ function such that $v=\phi$ on $\der\om$.

We shall denote by $\|\cdot\|_{\star}$ the norm in $\mathcal{L}\left(H^{1/2}_{co}(\Sigma),H^{-1/2}_{co}(\Sigma)\right)$
 defined by
\[
\|T\|_{\star}=\sup \{<T\psi, \phi> | \,\psi,\phi \in H^{1/2}_{co}(\Sigma), \hbox{ }
\|\psi\|_{H^{1/2}_{co}(\Sigma)}=\|\phi\|_{H^{1/2}_{co}(\Sigma)}=1\}
\]
for every $T\in \mathcal{L}\left(H^{1/2}_{co}(\Sigma),H^{-1/2}_{co}(\Sigma)\right)$.

\bigskip

We finally recall an extension to systems of Alessandrini's identity \cite{Al1}, \cite{Is}. Let $u_1$ and $u_2$ be the solutions to
\[
 \mbox{div}(\CC^{k}  \nabhat u_k) =  0\mbox{ in }\om
\]
for $k=1,2$ respectively and with $\CC^{k}$, $k=1,2$, satisfying (\ref{strconv}).  Then
\begin{equation}\label{AlessId}
\int_{\om}(\CC^{1}-\CC^{2})\nabhat u_1:\nabhat u_2=<(\Lambda_{\CC^{1}}-\Lambda_{\CC^{2}})u_2,u_1>
\end{equation}
where $\Lambda_{\CC^{1}}$, $\Lambda_{\CC^{2}}$ denote the Dirichlet to Neumann map corresponding to $\CC^{1}$, $\CC^{2}$ respectively.

\subsection{Main assumptions and statement of the main result}\label{subsec2.3}
Let $A$, $L$, $\alpha_0$, $\beta_0$, $N$ be given positive numbers such that $N\in \NN$, $\alpha_0\in (0, 1)$, $\beta_0\in(0,2)$ and $L\geq 1$. We shall refer to them as the a priori data. Let $r_0$ be a positive number.

Our main assumptions are:

\bigskip

\noindent{\bf (A1)} $\om\subset\RR^3$ is an open bounded domain such that
\[
|\om|\leq Ar_0^3,
\]
and we assume that
\[\ov{\om}=\cup_{j=1}^N \ov{D}_j,\]
where  $D_j$, $j=1,\ldots,N$ are connected and pairwise non overlapping open domains of Lipschitz class with constants $r_0$, $L$.

We also assume that there exists one region, say $D_1$ such that $\der D_1\cap\der\om$ contains an open flat portion $\Sigma$ and that
for every $j\in\{2,\ldots,N\}$ there exist $j_1,\ldots,j_M\in\{1,\ldots,N\}$ such that
\[
D_{j_1}=D_1,\quad D_{j_M}=D_j,
\]
and, for every $k=2,\ldots,M$
\[
\der D_{j_{k-1}}\cap \der D_{j_k}
\]
contains a flat portion $\Sigma_k$ such that
\[
\Sigma_k\subset\om,\quad \forall k=2,\ldots,M.
\]

Furthermore, for $k=1,\ldots,M$, we assume  there exists $P_k\in\Sigma_k$ and a rigid transformation of coordinates such that $P_k=0$ and
\begin{eqnarray*}
  \Sigma_k\cap Q_{r_0/3,L} &=& \{x\in Q_{r_0/3,L}\,:\, x_3=0\}, \\
  D_{j_k}\cap Q_{r_0/3,L} &=& \{x\in  Q_{r_0/3,L}\,:\, x_3<0\}, \\
  D_{j_{k-1}}\cap  Q_{r_0/3,L} &=& \{x\in  Q_{r_0/3,L}\,:\, x_3>0\};
\end{eqnarray*}
where we set $\Sigma_1:=\Sigma$.

For simplicity we will call $D_{j_1},\ldots,D_{j_M}$ \textit{a chain of domains connecting $D_1$ to $D_j$}.
For any $k\in\{1,...M\}$ we will denote by $n_k$ the exterior unit vector to $\partial
D_k$ in $P_k$.

\bigskip

\noindent{\bf (A2)} We assume the tensor $\CC$ to be isotropic piecewise constant of the form
\begin{equation}\label{isotrelisa}
\CC=\sum_{j=1}^N\CC_j \chi_{D_j}(x)
\end{equation}
where
\[\CC_j=\lambda_j I_3\otimes I_3 +2\mu_j{\II}_{Sym}\]
with constant Lam\'{e} coefficients $\lambda_i$ and $\mu_i$ that satisfy
\begin{equation}\label{strongconv}
\alpha_0\leq \mu_j\leq \alpha_0^{-1},\quad \lambda_j\leq \alpha_0^{-1},\quad 2\mu_j+3\lambda_j\geq\beta_0,\quad j=1,\dots, N.
\end{equation}
For $j=1,\dots, N$, we denote the Poisson's ratio by
$\nu_j=\frac{\lambda_j}{2\left(\lambda_j+\mu_j\right)}$.
Note that each $\nu_j$ satisfies (\ref{boundPoisson}).

\bigskip

In the sequel we will introduce a number of constants that we will always denote by $C$. The values of these constants might
differ from one line to the other.

\bigskip

%
\begin{theo}\label{teo2.1}
%
Let $\om$ and $\Sigma$ satisfy assumption {\bf (A1)}. Then there exists a positive constant $C$ depending on $L, A, N,\alpha_0,\beta_0$ only such that, for any $\CC^{k}$, $k=1,2$ satisfying assumption {\bf (A2)}, we have
\begin{equation}\label{stabil}
\|\CC^{1}-\CC^{2}\|_{\infty}\leq C\|\Lambda^{\Sigma}_{\CC^1}-\Lambda^{\Sigma}_{\CC^2}\|_{\star}.
\end{equation}
\end{theo}
A better evaluation of constant $C$ is given in Remark \ref{bomba}.

\bigskip

In order to prove Theorem \ref{teo2.1} we will first state it in terms of  the forward map that maps Lam\'{e} parameters to the corresponding Dirichlet to Neumann map.  Then, we will apply to the forward map the following general result:

\begin{prop}\label{propBV}
Let $M_1$ and $M_2$ be positive numbers and $d\in \mathbb{N}$. Let $\mathcal{A}$
and ${\bold K}$ be an open subset and a compact subset of $\mathbb{R}^{d}$
respectively. Assume that ${\bold K}\subset\mathcal{A}$,

\[\mbox{dist}\left( {\bold K},\mathbb{R}^{d}\setminus \mathcal{A}\right) \geq
M_1,\mbox{ and }{\bold K}\subset B_{M_2}(0).\]

 Let $\mathcal{B}$ be a Banach space and let
$F:\mathcal{A}\to\mathcal{B}$ be such that:
\begin{enumerate}[(i)]
\item $F$ is Frech\'{e}t differentiable;
\item the Frech\'{e}t derivative  $F^\prime:\mathcal{A}\to \mathcal{L}(\RR^d,\mathcal{B})$ is uniformly continuous with a modulus of continuity $\sigma_1(\cdot)$;
\item $F_{|_{{\bold K}}}$ is injective;
\item $(F_{|_{{\bold K}}})^{-1}:F({\bold K})\to {\bold K}$ is uniformly continuous with a modulus of continuity $\sigma_2(\cdot)$;
\item $F^\prime$ is injective in ${\bold K}$, namely there is a positive number $q_0$ such that
    \[ \min_{x\in {\bold K},\\ |h|=1}\left\|F^\prime(x)[h]\right\|_{\mathcal{B}}\geq q_0;\]
\end{enumerate}
then we have
\[\|x_1-x_2\|_{\RR^d}\leq C\|F(x_1)-F(x_2)\|_{\mathcal{B}}\quad \mbox{ for every } x_1,x_2\in {\bold K},\]
where $C=\max\{\frac{2M_1}{\sigma_2^{-1}(\delta_1)},\frac{2}{q_0}\}$,  for $\delta_1=\frac{1}{2}\min\{\delta_0, M_2\}$ with $\delta_0=\sigma_1^{-1}(\frac{q_0}{2})$.
\end{prop}

This proposition holds also in infinite dimensional spaces.
For a proof of Proposition \ref{propBV} in finite dimensional space we refer to \cite[Prop.5 ]{BaVe}.

\bigskip

Let us now introduce the forward map corresponding to our problem. In order to represent the set of Lam\'{e} parameters, we will use the following notation: let $\underline{L}:=(\lambda_1,\dots,\lambda_N,\mu_1,\ldots,\mu_N)$ denote a vector in $\RR^{2N}$ and denote by $\mathcal{A}$ the open subset of $\RR^{2N}$ defined by
\begin{equation}\label{mathcalA}
\mathcal{A}:=\{\underline{L}
\in \RR^{2N} : \mu_j>0,\,\, 2\mu_j+3\lambda_j>0,\,\, j=1,\dots,N\}.
\end{equation}
For each vector $\underline{L}\in \mathcal{A}$ we can define a piecewise constant isotropic elastic tensor
$\CC_{\underline{L}}$  (as in (\ref{isotrelisa}))  with Lam\'e parameters $\lambda_j$ and $\mu_j$ for $j=1,\dots,N$.

In this case, $\|\CC_{\underline{L}}\|_{\infty}$ is equal to the norm in $\RR^{2N}$ given by
\[\|\underline{L}\|_{\infty}=\max_{j=1,\ldots,N}\left\{\max\{|\lambda_j|,|\mu_j|\}\right\}.\]

The (nonlinear) forward map can be defined as follows:
\begin{definition}\label{forwardmap}
Let $\om$ and $\Sigma$  satisfy  assumption {\bf (A1)}. \\Let us define $F: \mathcal{A}\rightarrow \mathcal{L}( H^{1/2}_{co}(\Sigma), H^{-1/2}_{co}(\Sigma))$ by
\[
F(\underline{L})=\Lambda_{\CC_{\underline{L}}}^{\Sigma}.
\]
\end{definition}
We can identify $F$ with $\tilde F:\mathcal{A}\rightarrow \mathcal{B}$ such that $\tilde F(\underline{L})=\tilde
\Lambda_{\CC_{\underline{L}}}^{\Sigma}$ (defined in (\ref{DNbilineare})), where $\mathcal{B}$ is the Banach space of bilinear form on $H^{1/2}_{co}(\Sigma)\times H^{1/2}_{co}(\Sigma)$ with the standard norm.

Let $\psi$ and $\phi\in H^{1/2}_{co}(\Sigma)$ and let $u_{\underline{L}}$ be the solution to
\[
    \left\{\begin{array}{rcl}
             \mbox{div}(\CC_{\underline{L}}  \nabhat u)& = & 0\mbox{ in }\om \\
             u & = & \psi\mbox{ on }\der\om,
           \end{array}\right.
\]
and $v_{\underline{L}}$ solution to
\[
    \left\{\begin{array}{rcl}
             \mbox{div}(\CC_{\underline{L}}  \nabhat v)& = & 0\mbox{ in }\om \\
             v & = & \phi\mbox{ on }\der\om,
           \end{array}\right.
\]
then
\[\tilde F(\underline{L})(\psi,\phi)=<F(\underline{L})(\psi),\phi>=\int_{\om}\CC_{\underline{L}}\nabhat u_{\underline{L}}:\nabhat v_{\underline{L}}.\]

\noindent In the sequel, we will write $F$ and $\Lambda_{\CC_{\underline{L}}}^{\Sigma}$ instead of $\tilde F$ and $\tilde \Lambda_{\CC_{\underline{L}}}^{\Sigma}$.

\bigskip

With the above notation, Theorem \ref{teo2.1}, can be stated as follows:
\begin{theo}\label{teoF}
Let $\om$ and $\Sigma$ satisfy assumption  {\bf (A1)} and let  ${\bold K}\subset\mathcal{A}$ be the compact subset
\begin{equation}\label{mathcalK}
{\bold K}:=\{\underline{L}\in\mathcal{A}  : \alpha_0\leq\mu_j\leq \alpha_0^{-1}, \lambda_j\leq \alpha_0^{-1},  2\mu_j+3\lambda_j\geq \beta_0, j=1,\dots,N\}.
\end{equation}
Then, there exists a positive constant  $C$, depending on $L,A, N,\alpha_0,\beta_0$ only such that
\[
\|\underline{L}^{1}-\underline{L}^{2}\|_{\infty}\leq C \|F(\underline{L}^{1})-F(\underline{L}^{2})\|_{\star}
\]
for every $\underline{L}^{1}$, $\underline{L}^{2}$ in ${\bold K}$.
\end{theo}

Notice that Theorem \ref{teoF} means that $F$ is invertible on ${\bold K}$ and its inverse is Lipschitz continuous.

In Section \ref{sec5} we will show that the forward map of definition \ref{forwardmap} satisfies all the assumptions of Proposition \ref{propBV}
for $\mathcal{A}$ and ${\bold K}$ defined as in (\ref{mathcalA}) and (\ref{mathcalK}) respectively and $\mathcal{B}$ is the space of bilinear form on $H^{1/2}_{co}(\Sigma)\times H^{1/2}_{co}(\Sigma)$. Then, Theorem \ref{teoF} is a consequence of Proposition \ref{propBV}.

\section{Preliminary results}\label{prelres}
\subsection{Further notation and definitions}
In order to prove the main theorem we need to introduce some further notation and definitions.

\textbf{Construction of an augmented domain and extension of $\CC$.}

\noindent First we extend the domain $\Omega$ to a new domain $\Omega_0$ such that $\partial \Omega_0$ is of Lipschitz class and $B_{r_0/C}(P_1)\cap\Sigma\subset\Omega_0$, for some suitable constant $C\geq1$ depending only on $L$. We proceed as in \cite[Sect. 6]{A-R-R-V}.
Set
\begin{equation}
  \label{pho_1}
  \rho_1=r_0/C_L, \hbox{  where  } C_L=\frac{3\sqrt{1+L^{2}}}{L},
  \end{equation}
and define, for every $x'\in B'_{\frac{r_0}{3}}$
\[
 \psi^+(x')=
  \left\{ \begin{array}{ll}
  \frac{\rho_1}{2} &\mbox{for }
  |x'|\leq \frac{\rho_1}{4L},\\
  & \\
  \rho_1-2L|x'| &\mbox{for }
  \frac{\rho_1}{4L}< |x'|\leq \frac{\rho_1}{2L},\\
  & \\
  0 &\mbox{for }
  |x'|> \frac{\rho_1}{2L}.
  \end{array}\right.
\]
Observe that, for every $x'\in B'_{\frac{r_0}{3}}$,
$|\psi^+(x')|\leq \frac{\rho_1}{2}$ and  $|\nabla_{x'}\psi^+(x')|\leq 2L$.
Next, we denote by
\[
  D_0=\left\{x=(x',x_3) \in Q_{r_0/3,L}\quad | \quad
  0\leq x_{3}<\psi^+(x')
  \right\},
\]
\[
  \Omega_0=\Omega\cup D_0.
\]
It is straightforward to verify that
\begin{description}
\item{i)}
$\Omega_0$ has Lipschitz boundary with constants $\frac{r_0}{3}$, $3L$.
\item{ii)}
\[
  \Omega_0\supset Q_{r_0/4LC_L,L}.
\]
\end{description}

\bigskip

Let $\CC$ be an isotropic tensor that satisfies assumption {\bf (A2)}. We still denote by $\CC$ its  extension to $\Omega_0$ such that
$\CC_{|D_0}A=2\hat{A}$ for every $3\times3$ matrix $A$.
This extended tensor is still an isotropic elasticity tensor
of the form
\begin{equation}\label{isotr1}
\CC=\sum_{j=0}^N\CC_j \chi_{D_j}(x)
\end{equation}
where each $\CC_j$ for $j=0,\ldots,N$ has Lam\'{e} parameters satisfying \eqref{strongconv}.

\bigskip

\textbf{Construction of a walkway}

Let us fix $j\in\{1,...N\}$ and let
$D_{j_1},\ldots,D_{j_{M}}$ be a chain of domains connecting $D_1$ to $D_j$. For the sake of brevity set $D_k=D_{j_k}$, $k=1,...M$.

By \cite[Prop. 5.5]{A-R-R-V} there exists $C'_L\geq1$ depending on $L$ only, such that  $\left(D_{k}\right)_h$ is connected for every $k\in\{1,...M\}$ and
 every $h\in(0,r_0/C'_L)$.
Denote by
\begin{equation}
  \label{h_0}
  h_{0}=\min \left \{ \frac{r_{0}}{6},\frac{r_{0}}{C_{L}^{\prime}},\frac{\rho_1
}{8\sqrt{1+4L^{2}}}\right \},
\end{equation}
where $\rho_1$ is as in (\ref{pho_1}).

\bigskip

Let us introduce the following sets:
\begin{description}
\item{i)}
$Q_{(k)}$, $k=1,\ldots,M$, is the cylinder centered at $P_k$  such that by a rigid transformation of coordinates under which $P_k=0$ and $\Sigma_k$ belongs to
the plane $ \{(x',0)\}$, is given by $Q_{(k)}=Q_{\rho_1/4L,L}$. Moreover we denote $Q^{-}_{(M)}=Q_{(M)}\cap D_{M-1}$;
\item{ii)}
$\mathcal{K}$ is the interior part of the set $\bigcup_{i=0}^{M-1} \overline{D}_{i}$;
\item{iii)}
$\mathcal{K}_{h}=\bigcup_{i=0}^{M-1}\left(D_{i}\right)_h$, for every $h\in(0,h_0)$;
\item{iv)}
\begin{equation}
  \label{Corridoio}
  \widetilde{\mathcal{K}}_{h}=\mathcal{K}_{h}\cup Q^-_{(M)}\cup\bigcup_{k=1}^{M-1}Q_{(k)};
\end{equation}
\item{v)}
\[
  K_0=\left\{x \in D_0\quad | \quad
  dist(x,\partial\Omega)>\frac{\rho_1}{8}  \right\}.
\]
\end{description}
It is straightforward to verify that $\mathcal{K}_{h}$ is connected and of Lipschitz class for every $h\in(0,h_0)$ and that (in a suitable coordinate system)
\begin{equation}
  \label{palla in K_0}
  K_{0}\supset B_{\rho _{1}/4L}^{\prime }(P_{1})\times \left( \frac{\rho _{1}}{8}
  ,\frac{\rho _{1}}{4}\right).
\end{equation}
\subsection{Existence of singular solutions}\label{sec4}
\subsubsection{Fundamental solution in the biphase laminate}\label{subsec4.e}
In our proof of Lipschitz stability estimates, as in the approach used by
Alessandrini and Vessella for the conductivity equation \cite{AV}, a crucial
role is played by singular solutions for the Lam\'e system. As a matter of fact we are not only interested in the existence of such singular solutions, but also in their asymptotic behavior close to the interfaces.

In the scalar case, this tool is granted by the existence of  Green functions and by explicit expressions for solutions in the presence of an interface.

Unfortunately, the existence of the Green matrix cannot be inferred for elliptic systems with bounded coefficients in dimension 3 or higher.

Moreover, whereas for the scalar (even complex valued) conductivity equation is fairly easy to get a solution of the equation close to a flat interface by using the fundamental solution for the Laplace equation and suitable reflection arguments, this seems not to be possible for the Lam\'e system.

In order to construct singular solutions, we make use of special fundamental solutions constructed by Rongved in \cite{R} for  isotropic  biphase laminates.

Consider the isotropic tensor
\[\CC_b=\CC^+\chi_{\RR^3_+}+\CC^-\chi_{\RR^3_-}\]
where
$\CC^+$ and $\CC^-$ are constant isotropic elastic tensors given by
\[\CC^+=\lambda I_3\otimes I_3 +2\mu{\II}_{sym},\quad\CC^-=\lambda^\prime I_3\otimes I_3 +2\mu^\prime{\II}_{sym},\]
with $\lambda$ and $\mu$ and $\lambda^\prime$ and $\mu^\prime$ satisfy (\ref{strongconv}).
Denote Poisson's parameters by $\nu$ and $\nu^\prime$.

In \cite{R} an explicit formula for a fundamental solution $\Gamma:\{(x,y) \,:\,x\in\RR^3, y\in \RR^3, x\neq y\}\to \RR^{3\times 3}$ of
\[
\mbox{div}\left(\CC_b\nabhat \Gamma(\cdot,y)\right)=\delta_yI_3,
\]
is given. Here $\delta_y$ is the Dirac distribution concentrated at $y$.

This explicit formula is quite involved and some alternative formulations in more convenient tensor form have been proposed, for example in \cite{MeRe}.

For the purpose of the present work we need to point out some properties of biphase fundamental solution.
First of all, it is a fundamental solution, in the sense that
$\Gamma(x,y)$, is continuous in $\{(x,y)\in
\RR^3\times\RR^3:x\neq y\}$ and $\Gamma(x,\cdot)$ is locally integrable in $\RR^3$, for all $x\in \RR^3$, and,
 for every vector valued function $\phi\in C^{\infty}_0(\RR^3)$, we have
\[
\int_{\RR^3}\CC_b\nabhat\Gamma(\cdot,y):\nabhat\phi=\phi(y).
\]
Furthermore, for every $x,y\in\RR^3$, $x\neq y$, we have
\begin{equation}\label{sym}
\Gamma(x,y)=\Gamma(y,x)^T,
\end{equation}
\[
|\Gamma(x,y)|\leq \frac{C}{|x-y|},
\]
and, for any $r>0$,
\begin{equation}\label{boundrongved2}
\|\nabla \Gamma(\cdot,y)\|_{L^2(\RR^3\backslash B_r(y))}\leq \frac{C}{(rr_0^3)^{1/2}},
\end{equation}
where $C$ depends on $\alpha_0,\beta_0$ only.

We will also need to use the  explicit representation of  the some components of biphase  fundamental solution.
In particular, we will make use of explicit expression of the third column of $\Gamma$.
For $x=(x_1,x_2,x_3)$ with $x_3>0$ and $y=(0,0,c)$ with $c>0$, from \cite{R} we have,
\begin{equation}\label{rongelisa}
\Gamma(x,y)\cdot e_3=\frac{3-4\nu}{4(1-\nu)}\left(
\begin{array}{c}0\\0\\B\end{array}\right)
-\frac{1}{4(1-\nu)}\left[x_3\left(\begin{array}{c}B_{x_1}\\B_{x_2}\\B_{x_3}\end{array}\right)+\left(
\begin{array}{c}\beta_{x_1}\\ \beta_{x_2} \\ \beta_{x_3}\end{array}\right)\right],
\end{equation}
where
\[B=\frac{1}{4\pi\mu}\left\{\frac{1}{R_1}+\alpha \left[\frac{3-4\nu}{R_2}+\frac{2c(x_3+c)}{R_2^3}\right]\right\},\]
\[\beta=-\frac{1}{4\pi\mu}\left\{\frac{c}{R_1}+\alpha\left[\frac{c(3-4\nu)}{R_2}-\gamma\log(R_2+x_3+c)\right]\right\},\]
with
\begin{equation}\label{alpha}
\alpha=F_1(\mu,\mu^\prime,\nu):=\ds{\frac{\mu-\mu^\prime}{\mu+(3-4\nu)\mu^\prime}},
\end{equation}
\begin{equation}\label{gamma}
\gamma=F_2(\mu,\mu^\prime,\nu,\nu^\prime):=\ds{\frac{4(1-\nu)\mu\left[(1-2\nu)
(3-4\nu^\prime)-2\frac{\nu-\nu^\prime}{\mu-\mu^\prime}\mu^\prime\right]}
{\mu^\prime+(3-4\nu^\prime)\nu},}
\end{equation}
\[R_1=\left(x_1^2+x_2^2+(x_3-c)^2\right)^{1/2},\mbox{ and }R_2=\left(x_1^2+_2^2+(x_3+c)^2\right)^{1/2}.\]
It is worth noticing that, for $\CC^+=\CC^-$, the matrix $\Gamma$ coincides with the free space fundamental matrix for constant isotropic elasticity tensor.

\subsubsection{Singular solutions}\label{subsec4.3}
With the aid of the biphase fundamental solution $\Gamma$, let us construct singular solutions with singularities in some subset of the domain. Since $\Gamma$ is defined only in the case of a flat interface, we will need to keep away from interfaces that are not flat.

For this reason we need to set the following notation: let $\mathcal{F}$ be the union of the flat parts of  $\cup_{j=0}^N\der D_j$. By flat parts we intend that they can be represented as the graphs of a constant function in at least a ball of radius $\frac{r_0}{3}$ (as $\Sigma_k$ in assumption \textbf{(A1)}). Let  $\mathcal{D}=\cup_{j=0}^N\der D_j\setminus \mathcal{F}$. The set $\mathcal{D}$ contains the non flat parts of the interfaces.

 Let $\CC=\sum_{j=0}^N\CC_j\chi_{D_j}$ with tensors $\CC_j$  satisfying \textbf{(A2)} for all $j$.
 Let $y\in \om_0\backslash\mathcal{D}$ and let $r=\min (r_0/4,\textrm{dist}(y,\mathcal{D}\cup \der\om_0))$. Then, in the sphere $B_r(y)$ either $\CC$ is constant, $\CC=\CC_j$ or, by a suitable choice of the coordinate system,
 $\CC=\CC_j+(\CC_{j+1}-\CC_j)\chi_{\{x_3>a\}}$ for some $j=0,1,\cdots,N$ and some $a$ with $|a|<r$.
Let
\[
   \CC_y= \left\{\begin{array}{rcl}
            &\CC_j \quad&\text{ if }\CC=\CC_j \text{ in }B_r(y),\\
             & \CC_j+(\CC_{j+1}-\CC_j)\chi_{\{x_3>a\}}& \text{ otherwise},
           \end{array}\right.
\]
and consider the biphase fundamental solution to
\[
\mbox{div}( \CC_y\nabhat \Gamma(\cdot, y))=\delta_y I_3\text{ in }\RR^3.
\]
\begin{prop}\label{Green}
Let $\om_0$ and  $\CC$ satisfy assumptions \textbf{(A1)} and \textbf{(A2)}.
Then, for $y\in \om_0\backslash\mathcal{D}$, there exists a unique function $G(\cdot,y)$, continuous in $\om\setminus\{y\}$
such that
\[
\int_{\om_0}\CC\nabhat G(\cdot,y):\nabhat \phi=\phi(y),\mbox{ for every }\phi\in C^\infty_0(\om_0).
\]
and such that
\[
G(\cdot,y)=0\quad \text{ on }\der\om_0.
\]
Furthermore, if $\text{ dist}(y, \mathcal{D}\cup\der\om_0))\geq r_0/c_1$ for some $c_1>1$,
\begin{equation}\label{energy}
\|G(\cdot,y)-\Gamma(\cdot, y)\|_{H^1(\om_0)}\leq Cr_0^{-1}
\end{equation}
\begin{equation}\label{green1}
    \|G(\cdot,y)\|_{H^1(\om_0\setminus B_r(y))}\leq C (rr_0)^{-1/2},
\end{equation}
where $C$ depends $\alpha_0$, $\beta_0$, $A$, $L$ and, increasingly, on $c_1$.
Furthermore,
\begin{equation}\label{symgreen}
G(x,y)=G(y,x)^T\quad\mbox{for every }x,y\in \om_0\backslash\mathcal{D}.
\end{equation}
\end{prop}
\begin{proof}
Let us set
\[
G(x,y):=\Gamma(x,y)+w(x,y),
\]
where $w$ is solution to
\begin{equation}\label{w}
   \left\{\begin{array}{rcl}
             \mbox{div}\left(\CC\nabhat_x w(\cdot,y)\right)&= &\mbox{div}((\CC-\CC_y)\nabhat \Gamma(\cdot,y))\quad\mbox{in}\quad\om_0 \\
             w(\cdot,y)&= &-\Gamma(\cdot,y)  \quad\mbox{on}\quad\der\om_0.\\
           \end{array}
    \right.
 \end{equation}
Since $\CC-\CC_y=0$ in $B_r(y)$ and $\Gamma(\cdot,y)$ is
smooth and bounded in $\om_0\backslash B_r(y)$, then $f=(\CC-\CC_y)\nabhat \Gamma(\cdot,y)\in L^2(\om_0)$, $\mbox{div} f\in H^{-1}(\om_0)$ and $\Gamma(\cdot,y)|_{\der\om_0}\in H^{1/2}(\der\om_0)$, hence this problem has a unique solution $w\in H^1(\om_0)$.

Furthermore, if $\text{ dist}(y, \mathcal{D}\cup\der\om_0))\geq r_0/c_1$ for some $c_1>1$,
\[
\|w(\cdot,y)\|_{H^1(\om_0)}\leq C(\|f(\cdot,y)\|_{H^{-1}(\om_0)}+\|\Gamma(\cdot,y)\|_{H^{1/2}(\der\om_0)})\leq C
r_0^{-1}\]
where $C$ depends on $\alpha_0$, $\beta_0$ $A$ and $c_1$. This proves (\ref{energy}).
Estimate (\ref{green1}) follows from (\ref{energy}) and from (\ref{boundrongved2}).

The symmetry of function $G$ follows by standard arguments of potential theory \cite{E}.
\end{proof}
\subsubsection{Two useful properties of the biphase fundamental solution.}
Let $\CC_b$, $\overline{\CC}_b$ be given by
\[\CC_b=\CC^+\chi_{\RR^3_+}+\CC^-\chi_{\RR^3_-},\quad\overline{\CC}_b=\CC^+\chi_{\RR^3_+}+\overline{\CC}^{\,-}\chi_{\RR^3_-},\]
where $\CC^+$, $\CC^-$ and $\overline{\CC}^{\,-}$ are constant and strongly convex isotropic tensors whose Lam\'e  coefficients satisfy (\ref{strconvlame})

Let $\Gamma_{\CC_b}$ and $\Gamma_{\overline{\CC}_b}$ be the biphase fundamental solutions relative to operators  $\mbox{div}\left(\CC_b\nabhat\cdot\right)$ and
$\mbox{div}\left(\overline{\CC}_b\nabhat\cdot\right)$, respectively.
\begin{prop}\label{trick}
 For every $l,m\in\mathbb{R}^3$ and every $y,z\in\mathbb{R}^3_{+}$, $y\neq z$ we have
\[
 \int_{\mathbb{R}^3_{-}}\!\!\!\!
 \left(\CC_b-\overline{\CC}_b\right)\nabhat\Gamma_{\CC_b}(\cdot,y)\,l:
 \nabhat\Gamma_{\overline{\CC}_b}(\cdot,z)\,m
 =\left(\Gamma_{\CC_b}(y,z)-\Gamma_{\overline{\CC}_b}(y,z)\right)m\cdot l.
\]
\end{prop}

\begin{proof}
Let us fix $l,m\in\mathbb{R}^3$ and $y,z\in\mathbb{R}^3_{+}$, $y\neq z$ and denote by
\[
u=\Gamma_{\CC_b}(\cdot,y)\,l \quad\mbox{, } v=\Gamma_{\overline{\CC}_b}(\cdot,z)\,m.
\]
Let $\varepsilon_0>0$ such that $B_{\varepsilon_0}(y),B_{\varepsilon_0}(z)\subset \mathbb{R}^3_{+}$ and $B_{\varepsilon_0}(y)\cap
B_{\varepsilon_0}(z)=\emptyset$. Since $\CC_b=\overline{\CC}_b$ in $\mathbb{R}^3_{+}$, we get trivially
\[
 \int_{\mathbb{R}^3_{+}\setminus\left(B_{\varepsilon}(y)\cup B_{\varepsilon}(z)\right)}\!\!\!\!\!\!\!\!\!\!\!\!\left(\CC_b-\overline{\CC}_b\right)\nabhat u:\nabhat v
 =0, \quad\mbox{for every } \varepsilon\in(0,\varepsilon_0).
\]
Hence
\begin{equation}\label{Spsi}
 \int_{\mathbb{R}^3_{-}}\!\!\!\!  \left(\CC_b-\overline{\CC}_b\right)\nabhat u:
 \nabhat v =\int_{\mathbb{R}^3\setminus\left(B_{\varepsilon}(y)\cup B_{\varepsilon}(z)\right)}\!\!\!\!\!\!\!\!\!\!\!\!\left(\CC_b-\overline{\CC}_b\right)\nabhat u:\nabhat v,
 \quad\mbox{for every } \varepsilon\in(0,\varepsilon_0).
\end{equation}
Integration by parts in $\mathbb{R}^3\setminus\left(B_{\varepsilon}(y)\cup B_{\varepsilon}(z)\right)$ yields, for every $\varepsilon\in(0,\varepsilon_0)$,
\begin{equation}\label{Sintegral}
\int_{\mathbb{R}^3\setminus\left(B_{\varepsilon}(y)\cup B_{\varepsilon}(z)\right)}\!\!\!\!\!\!\!\!\!\!\!\!\!\!\!\!\CC_b\nabhat u:\nabhat v=-\int_{\der
B_{\varepsilon}(y)}\left(\CC_b\nabhat u\right)n\cdot v
-\int_{\der B_{\varepsilon}(z)}\left(\CC_b\nabhat u\right)n\cdot v,
\end{equation}
where $n$ is the outward unit normal to $\partial \left(B_{\epsilon}(y)\cup B_{\varepsilon}(z)\right)$.
Now by definition of fundamental solution we have
\begin{equation}\label{Slimit1}
\lim_{\varepsilon\rightarrow 0}\int_{\der B_{\epsilon}(y)}\!\!\left(\CC_b\nabhat u\right)n\cdot v=v(y)\cdot l
\,\,\mbox{ and }\,\,\,
\lim_{\varepsilon\rightarrow 0}\int_{\der B_{\epsilon}(z)}\!\!\left(\CC_b\nabhat u\right)n\cdot v=0.
\end{equation}
By \eqref{Sintegral} and \eqref{Slimit1} we have
\begin{equation}\label{Slimit3}
\lim_{\varepsilon\rightarrow 0}\int_{\mathbb{R}^3\setminus\left(B_{\varepsilon}(y)\cup B_{\varepsilon}(z)\right)}\!\!\!\!\!\!\!\!\!\!\!\!\CC_b\nabhat u:\nabhat v=-v(y)\cdot l.
\end{equation}
In a similar way,  we get
\begin{equation}\label{Slimit4}
\lim_{\varepsilon\rightarrow 0}\int_{\mathbb{R}^3\setminus\left(B_{\varepsilon}(y)\cup B_{\varepsilon}(z)\right)}\!\!\!\!\!\!\!\!\!\!\!\!\overline{\CC}_b\nabhat u:\nabhat v=-u(z)\cdot m.
\end{equation}
By \eqref{Spsi}, \eqref{Slimit3}, and \eqref{Slimit4} we have
\begin{equation*}
 \int_{\mathbb{R}^3_{-}}\!\!\!\!  \left(\CC_b-\overline{\CC}_b\right)\nabhat u:
 \nabhat v =u(z)\cdot m-v(y)\cdot l=\Gamma_{\CC}(z,y)\,l\cdot m - \Gamma_{\overline\CC}(y,z)\,m\cdot l.
\end{equation*}
  By \eqref{sym} the thesis follows.
\end{proof}

\begin{prop}\label{trick2}
Let $h,k$ be real numbers and let ${\HH}$ be the fourth order tensor
\[
{\HH}(x) =\left(h\, I_3\otimes I_3 +2k\,{\II}_{sym}\right)\chi _{\mathbb{R}
_{-}^{3}}(x).
\]
For every $l,m\in\mathbb{R}^3$ and every $y,z\in\mathbb{R}^3_{+}$, $y\neq z$, we have
\[
 \int_{\mathbb{R}^3_{-}}\!\!\!\!\HH\nabhat\Gamma_{\CC_b}(\cdot,y)\,l:\nabhat\Gamma_{\CC_b}(\cdot,z)\,m
 =\left(\frac{d}{dt}\Gamma_{\CC_b+t\HH}(y,z)m\cdot l\right)_{|_{t=0}}\!\!\!\!\!.
\]
\end{prop}
\begin{proof}
Let us fix $l,m\in\mathbb{R}^3$ and $y,z\in\mathbb{R}^3_{+}$, $y\neq z$. Let $t_0$ be a positive number such that for every $t\in(-t_0,t_0)$
the tensor $\CC_b+t\HH$ is strongly convex.

Since $\HH(x)=0$ for  every $x\in\mathbb{R}^3_{+}$ we have trivially

\[
 \psi^{(l,m)}(y,z):=\int_{\mathbb{R}^3_{-}}\!\!\!\!\HH\nabhat\Gamma_{\CC_b}(\cdot,y)\,l:\nabhat\Gamma_{\CC_b}(\cdot,z)\,m=
 \int_{\mathbb{R}^3}\!\!\!\!\HH\nabhat\Gamma_{\CC_b}(\cdot,y)\,l:\nabhat\Gamma_{\CC_b}(\cdot,z)\,m.
\]

Hence, for every $t\in(-t_0,t_0)\setminus\{0\}$,
\begin{eqnarray}\label{psitrickS2}
 \psi^{(l,m)}(y,z)=\frac{1}{t}\int_{\mathbb{R}^3}\left((\CC_b+t\HH)-\CC_b\right)\nabhat\Gamma_{\CC_b}(\cdot,y)\,l:\nabhat\Gamma_{\CC_b}(\cdot,z)\,m=\\
\nonumber
=\frac{1}{t}\int_{\mathbb{R}^3}\left((\CC_b+t\HH)-\CC_b\right)\nabhat\Gamma_{\CC_b+t\HH}(\cdot,y)\,l:\nabhat\Gamma_{\CC_b}(\cdot,z)\,m-\\
-\int_{\mathbb{R}^3_{-}}\HH\left(\nabhat\Gamma_{\CC_b+t\HH}(\cdot,y)\,l-\nabhat\Gamma_{\CC_b}(\cdot,y)\,l\right):\nabhat\Gamma_{\CC_b}(\cdot,z)\,m.
\nonumber
\end{eqnarray}
By (\ref{psitrickS2}) and by Proposition \ref{trick} we get
\begin{eqnarray}\label{psitrickS3}
 \psi^{(l,m)}(y,z)&=&\frac{1}{t}\left((\Gamma_{\CC_b+t\HH}(y,z))-\Gamma_{\CC_b}(y,z)\right)\,m\cdot l\\
\nonumber
&-&\int_{\mathbb{R}^3_{-}}\HH\left(\nabhat\Gamma_{\CC_b+t\HH}(\cdot,y)\,l-\nabhat\Gamma_{\CC_b}(\cdot,y)\,l\right):\nabhat\Gamma_{\CC_b}(\cdot,z)\,m.
\end{eqnarray}
Now, by straightforward calculation on the biphase fundamental solution given in  \cite{R}, we have, for every $x\in\mathbb{R}^3_{-}$,
\begin{equation}\label{psitrickS4}
\lim_{t\rightarrow 0}
\left(\nabhat\Gamma_{\CC_b+t\HH}(x,y)\,l-\nabhat\Gamma_{\CC_b}(x,y)\,l\right)=0, \end{equation}
and, for every $t\in(-t_0,t_0)$,
\begin{equation}\label{psitrickS5}
\left \vert {\HH\left(\nabhat\Gamma_{\CC_b+t\HH}(x,y)\,l-\nabhat\Gamma_{\CC_b}(x,y)\,l\right):\nabhat\Gamma_{\CC_b}(x,z)\,m}\right \vert\leq C\min\{|x|^{-4},1\},
\end{equation}
where $C$ depends on $\alpha_0$, $\beta_0$, $y$, $z$, $l$ and $m$ only.
By \eqref{psitrickS4}, \eqref{psitrickS5}, and applying the dominated convergence theorem,  we get
\begin{equation}\label{psitrickS6}
\lim_{t\rightarrow 0}\int_{\mathbb{R}^3_{-}}\HH\left(\nabhat\Gamma_{\CC_b+t\HH}(.,y)\,l-\nabhat\Gamma_{\CC_b}(.,y)\,l\right):\nabhat\Gamma_{\CC_b}(.,z)\,m=0.
\end{equation}
Finally, by \eqref{psitrickS6} and \eqref{psitrickS3} the thesis follows.

\end{proof}

\subsection{Some estimates  for solutions to the Lam\'{e} system}\label{reg}
In this section we collect some properties concerning solutions to the linearized elasticity system with piecewise constant elasticity tensor that will be crucial to proe our main result.

First  we state  a regularity result for solutions of elliptic systems in composite materials from  \cite{LN} and \cite{CKVC}. Afterwords we use this result in order to obtain Proposition  \ref{Analytic}, which is then used to prove a quantitative estimate of unique continuation for solutions of  systems satisfying assumptions \textbf{(A1)} and \textbf{(A2)}.

\begin{prop}\label{regularity}
Let $\CC^+$ and $\CC^-$ be two isotropic elasticity tensors with constant Lam\'e coeffcients $\lambda$, $\mu$ and $\lambda^{\prime}$, $\mu^{\prime}$ respectively, satisfying assumption \textbf{(A2)}, let $R>0$ and $\CC_b=\CC^+\chi_{\RR^3_+}+\CC^-\chi_{\RR^3_-}$. 
Let $v\in H^1(B_R)$ be a weak solution to
\[
\mbox{div}(\CC_b\nabhat v ) =  0\mbox{ in }B_R.
\]
Then, for every multiindex $\beta^\prime$, $D_{x_\prime}^{\beta^\prime} v\in C^{0}(B_R)$ and $v\in C^{\infty}(\overline{B^{\pm}_R})$. Moreover for any $\delta>0$ and $k\geq 0$
\begin{equation}\label{regest}
R^k\|D^k v\|_{L^\infty(\overline{B^{\pm}}_{(1-\delta)R})}\leq CR^{-3/2}\left[\int_{B_R}|v|^2 dx\right]^{1/2}
\end{equation}
 where 
 $C=C(\delta,k,\alpha_0,\beta_0)$.
\end{prop}
Note that, by $v\in C^{\infty}\left(\overline{B^{\pm}_R}\right)$ we intend that $v_{|_{B_R^-}}$ has a $C^\infty$  extension to $\overline{B_R^-}$ and $v_{|_{B_R^+}}$ has a $C^\infty$  extension to $\overline{B_R^+}$.

Observe that, in particular,
we have
\begin{equation}\label{stimaLinfinito}
    \|v\|_{L^\infty(\overline{B}_{(1-\delta)R})}\!\!\leq CR^{-3/2}\left[\int_{B_R}|v|^2 dx\right]^{1/2}.
\end{equation}

\bigskip
Beside this regularity result, the principal ingredients for proving estimates of unique continuation are the three sphere inequality, some stability estimates for the Cauchy problem and a smallness propagation estimate in a cone. All these results holds for constant elasticity tensors and are precisely described in the Appendix.

\bigskip

The flatness assumption on interfaces allows us to get a better estimate for the Lipschitz constant. The reason is the fact that solutions to the system with piecewise constant elasticity tensor have analytic extension beyond the flat interface. This property is stated in the following Proposition.
\bigskip

Here we use the notation of section \ref{subsec4.e} and assumptions \textbf{(A1)} and \textbf{(A2)}.
\begin{prop}\label{Analytic}
Let $v\in H_{loc}^1(\mathcal{K})$ be a solution to
\begin{equation}
  \label{equazA}
\mbox{div}\left(\CC\widehat{\nabla}v\right)=0\quad\mbox{in}\quad \mathcal{K}.
\end{equation}
Let us fix $k\in\{0,\ldots,M\}$. Then there exist two positive constants $C_1$ and $C$, depending only on $\alpha_0$, $\beta_0$ and $L$, such that $v_{|D_k}$ can be
extended by a function $\widetilde{v}$ in the set $D_k\cup\Xi^{C_1}_{k+1}$, where
\begin{equation}
  \label{Xi}
    \Xi^{C_1}_{k+1}=\left \{ x\in D_{k}\cup D_{k+1}\cup \Sigma _{k}:\text{ }dist\left(
x,B_{r_{0}/6}\left( P_{k}\right) \cap \Sigma _{k}\right) <\frac{r_{0}}{4C_{1}
}\right \}
\end{equation}
and
\begin{equation}
\label{3.32}
    \|\tilde{v}\|_{L^\infty\left(\Xi^{C_1}_{k+1}\right)}\leq C\|v\|_{L^2(\Xi^{C_1/2}_{k+1})}.
    \end{equation}
\end{prop}

\textit{Proof of Proposition \ref{Analytic}}\\
 It is not restrictive to assume that $P_k=0$, $\Sigma_k$ belongs to the plane $ \{x_3=0\}$ and $B^+_{r_0/3}\subset D_k$, $B^-_{r_0/3}\subset D_{k+1}$.
 By Proposition \ref{regularity} we know that $v_{|D_k}\in C^\infty (D_k\cup\Sigma_k)$ and $v_{|D_{k+1}}\in C^\infty (D_{k+1}\cup\Sigma_k)$.
   Since $\CC$ is a constant tensor in each domain, for every
 $\beta^\prime\in \left(\mathbb{N}\cup\{0\}\right)^{2}$, $D^{\beta^\prime}_{x^\prime}v$ is a solution to \eqref{equazA}.
 \\
 Denote by $R=\frac{r_0}{6}$ and $v_{+}:=v_{|\overline{B}^{+}_{R}}$.
 For any $x_0$ in $\Sigma _{k}\cap B_{R}$ ,
 $D^{\beta^\prime}_{x^\prime}v\in C^{0}\left(B_R(x_0)\right)$   and $v_{|\overline{B}^{+}_R(x_0)}\in C^\infty(\overline{B}^{+}_R(x_0))$.

 Let us recall the following Caccioppoli inequality (\cite[pag.20]{BBFM}): let $u$ be a solution to  \eqref{equazA} then, for every $x_0\in\Sigma_k\cap B_R$
 \begin{equation}
  \label{Caccio}
\int_{B_{\rho_{2}}(x_0)}\left|\nabla u\right|^2\leq
\frac{C}{(\rho_{2}-\rho_{1})^2}\int_{B_{\rho_{1}}(x_0)}\left| u\right|^2
\end{equation}
for $0<\rho_{2}<\rho_{1}<R$ where $C$ depends on $\alpha_0$ and $\beta_0$ only.

Denote by
\begin{equation}\label{Notation}
\phi(x^\prime)=v_{+}(x^\prime,0) \quad\mbox{and  }  \psi(x^\prime)=\frac{\partial v_{+}}{\partial_{x_3}}(x^\prime,0).
\end{equation}
We now prove that $\phi$ and $\psi$ are analytic in $\Sigma _{k}\cap B_{R}$ and we estimate the derivatives $D^{\beta^\prime}_{x^\prime}\phi$ and $D^{\beta^\prime}_{x^\prime}\psi$ from above on $B_{R/4}(x_0)$ for every $\beta^\prime\in \left(\mathbb{N}\cup\{0\}\right)^{2}$ and for every $x_0\in\Sigma_k\cap B_R$.

Starting from \eqref{Caccio} and using the same iterative procedure followed to prove inequality (39) in
\cite{BF} we obtain
\[
    \sum_{|\beta|=N_0}\int_{B_{\frac{R}{2}}(x_0)}|D^{\beta^\prime}_{x^\prime} v|^2\leq \left(C\left(\frac{2N_0}{R}\right)^2\right)^{N_0}
    \int_{B_R(x_0)}|v|^2,
    \]
for every $N_0\in\mathbb{N}$, where $C$ depends on $\alpha_0$ and $\beta_0$ only.

On the other side, by Proposition \ref{regularity} we have
\[
\|D^{\beta^\prime}_{x^\prime} v\|_{L^\infty(B_{\frac{R}{4}}(x_0))}\leq
C\|D^{\beta^\prime}_{x^\prime}v\|_{L^2(B_{\frac{R}{2}}(x_0))}.
\]
Hence, proceeding as in \cite[(41)]{BF} and applying Proposition \ref{regularity} to $D^{\beta^\prime}_{x^\prime}v$, we get
\begin{equation}\label{A3.20}
    \|D^{\beta^\prime}_{x^\prime} \phi\|_{L^\infty(B^{\prime}_{\frac{R}{4}}(x_0))}\leq \beta^\prime!\left(\frac{C}{R}\right)^{|\beta^\prime|}\|v\|_{L^2(B_R(x_0))},
\end{equation}
and
\begin{equation}\label{A3.21}
    \|D^{\beta^\prime}_{x^\prime}\psi\|_{L^\infty(B^\prime_{\frac{R}{4}}(x_0))}\leq \beta^\prime!\left(\frac{C}{R}\right)^{|\beta^\prime|+1}\|v\|_{L^2(B_R(x_0))}.
\end{equation}
where $C$ depends on $\alpha_0$ and $\beta_0$ only.
  By \eqref{A3.20} and \eqref{A3.21} and by the Cauchy-Kowalevski theorem we have that the solution $\tilde{v}$ to the Cauchy problem
\[
    \left\{\begin{array}{rcl}
             \mbox{div}\left(\CC\widehat{\nabla}\tilde{v}\right)&=&0\\
             \tilde{v}(\cdot,0) & =& \phi\quad\mbox{ on }  \Sigma _{k}\cap B_{R}, \\
             \frac{\partial \tilde{v}}{\partial x_3}(\cdot,0)&=&\psi \quad\mbox{ on }  \Sigma _{k}\cap B_{R}
           \end{array}\right.
\]
is analytic in the neighborhood $\Xi^{C_1}_{k+1}$ of  $\Sigma _{k}\cap B_{R}$. Therefore, taking into account \eqref{Notation}, $\tilde{v}$  is the analytic
extension of $v_+$ in $\Xi^{C_1}_{k+1}$ and estimate \eqref{3.32} follows.\qed

\bigskip

 Finally we state  a quantitative estimate of unique continuation.
\begin{prop}\label{QEUC}
Let $\varepsilon_1$, $E_1$ and $h$ be positive numbers, $h<h_0$, where $h_0$ is defined in \eqref{h_0}. Let $v\in H_{loc}^1(\mathcal{K})$
be a solution to
\[
\mbox{div}\left(\CC\widehat{\nabla}v\right)=0\quad\mbox{in}\quad \mathcal{K},
\]
such that
\[
    \|v\|_{L^\infty(K_0)}\leq \varepsilon_1,
\]
and
\begin{equation}
  \label{limitaz}
|v(x)|\leq E_1\left(\frac{\mbox{dist}(x,\Sigma_{M})}{r_0}\right)^{-\frac{1}{2}}
    \quad\mbox{for every}\quad x\in \mathcal{K}_{h/2}.
\end{equation}
Then
\[
    |v(\tilde{x})|\leq C\left( \frac{r_{0}}{r}\right) ^{2}\varepsilon_1 ^{\widetilde{\theta}^{\overline{m}M}\tau_{r}}
    (E_1+\ep_1)^{1-\widetilde{\theta}^{\overline{m}M}\tau _{r}},
\]
where $r\in\left(0,\frac{r_0}{C}\right)$, $\tilde{x}=P_{M}+rn_{M}$,
\[
\tau_{r}=\widetilde{\theta}\left( \frac{r}{r_0}\right) ^{\delta}
\]
and $\overline{m}$, $C$, $\delta$, $\widetilde{\theta}$, $0<\widetilde{\theta}<1$, depend on $A$, $L$, $\alpha_0$, $\beta_0$ and $N$.
\end{prop}

\bigskip

The proof of the above Proposition is given in the Appendix.

\section{Proof of the main result}\label{sec5}
This section contains the proof of the main result that consists in showing that the forward map introduced in definition \ref{forwardmap} satisfies all the assumptions of Proposition \ref{propBV}.

\subsection{Differentiability of $F$}\label{subsec5.1}
 \begin{prop}\label{deriv}
The map
\[
F:\mathcal{A}\rightarrow \mathcal{L}( H^{1/2}_{co}(\Sigma), H^{-1/2}_{co}(\Sigma))
\]
defined in (\ref{forwardmap}) is Frech\'et differentiable in $\mathcal{A}$ and
\begin{equation}\label{differenziale}
<F^\prime(\underline{L})[\underline{H}]\psi,\phi>=\int_{\om}
\HH \nabhat u_{\CC}:\nabhat v_{\CC}
\end{equation}
where $\HH=\CC_{\underline{H}}$.

Moreover, $F^\prime:\mathcal{A}\rightarrow \mathcal{L}\left(\RR^{2N},\mathcal{L}( H^{1/2}_{co}(\Sigma), H^{-1/2}_{co}(\Sigma))\right)$ is Lipschitz continuous with Lipschitz constant $C_{F^\prime}$ depending on $A,L, \alpha_0,\beta_0$ only.
\end{prop}
\begin{proof}
Fix $\underline{L}\in\mathcal{A}$ and let $\underline{H}\in \RR^{2N}$ such that $\|\underline{H}\|_{\infty}$ is sufficiently small.
By (\ref{AlessId}) we have
\[
<\left(F(\underline{L}+\underline{H})-F(\underline{L})\right)\psi,\phi>=
\int_{\om}\HH\nabhat u_{\underline{L}+\underline{H}}:\nabhat v_{\underline{L}}.
\]
Hence, by setting,
\[
\eta:=<\left(F(\underline{L}+\underline{H})-F(\underline{L})\right)\psi,\phi>\!-\!\int_{\om}\!\HH\nabhat u_{\underline{L}}:\nabhat v_{\underline{L}}=\int_{\om}\!\HH\nabhat
(u_{\underline{L}+\underline{H}}-u_{\underline{L}}):\nabhat v_{\underline{L}},
\]
we have
\begin{equation}\label{3.3}
|\eta|\leq C r_0^2\|\underline{H}\|_{\infty}\|\nabla (u_{\underline{L}+\underline{H}}-u_{\underline{L}})\|_{L^{2}(\om)}\|\phi\|_{H^{1/2}_{co}(\Sigma)},
\end{equation}
where $C$ depends on $A,L, \alpha_0,\beta_0$ only. Let us estimate $\|\nabla (u_{\underline{L}+\underline{H}}-u_{\underline{L}})\|_{L^{2}(\om)}$. For, observe that
$w:=u_{\underline{L}+\underline{H}}-u_{\underline{L}}$ is solution to
\begin{equation}\label{3.4}
    \left\{\begin{array}{rcl}
             \mbox{div}(\CC_{\underline{L}}  \nabhat w)& = & \mbox{div}(\HH\nabhat u_{\underline{L}+\underline{H}})\mbox{ in }\om, \\
             w & = & 0\mbox{ on }\der\om,
           \end{array}\right.
\end{equation}
hence
\[
\int_{\om}\CC_{\underline{L}}\nabhat w:\nabhat w=\int_{\om}\HH\nabhat u_{\underline{L}+\underline{H}}:\nabhat w.
\]
By Lax-Milgram Theorem and Korn inequality we get
\begin{equation}\label{3.5}
\|\nabla w\|_{L^2(\om)}\leq C \|\underline{H}\|_{\infty}\|\nabla u_{\underline{L}+\underline{H}}\|_{L^{2}(\om)}\leq C r^{-1}_0
\|\underline{H}\|_{\infty}\|\psi\|_{H^{1/2}_{co}(\Sigma)},
\end{equation}
where $C$ depends on $A,L, \alpha_0,\beta_0$ only.

By inserting (\ref{3.5}) into (\ref{3.3}) we get
\begin{equation}\label{3.6}
|\eta|\leq C r_0 \|\underline{H}\|^2_{\infty}\|\psi\|_{H^{1/2}_{co}(\Sigma)}\|\phi\|_{H^{1/2}_{co}(\Sigma)},
\end{equation}
where $C$ on depends on $A,L, \alpha_0,\beta_0$ only, that yields (\ref{differenziale}).

Let us now prove the Lipschitz continuity of $F^\prime$. Let $\underline{L}^1$, $\underline{L}^2\in\mathcal{A}$ and set
\begin{eqnarray*}
\omega&:=&\!\!\!\!<\left(F^\prime(\underline{L}^2)-F^\prime(\underline{L}^1)\right)[\underline{H}]\psi,\phi>=\int_{\om}\HH \nabhat u_{\underline{L}^2}:\nabhat v_{\underline{L}^2}-\int_{\om}\HH
\nabhat u_{\underline{L}^1}:\nabhat v_{\underline{L}^1}\nonumber\\
&=&\!\!\!\!\int_{\om}\HH(\nabhat u_{\underline{L}^2}-\nabhat u_{\underline{L}^1}):\nabhat v_{\underline{L}^2}+\int_{\om}\HH\nabhat u_{\underline{L}^1}:(\nabhat
v_{\underline{L}^2}-\nabhat v_{\underline{L}^1}).
\end{eqnarray*}
By reasoning as we did to derive (\ref{3.6}) we obtain
\[
|\omega|\leq C_{F^\prime}r_0\|\underline{H}\|_{\infty}\|\underline{L}^2-\underline{L}^1\|_{\infty}\|\psi\|_{H^{1/2}_{co}(\Sigma)}\|\phi\|_{H^{1/2}_{co}(\Sigma)},
\]
where $C_{F^\prime}$ depends only on $A$, $L$, $\alpha_0$, and $\beta_0$.
\end{proof}

\subsection{Injectivity of $F_{|\mathbf{K}}$ and uniform continuity of  $(F_{|\mathbf{K}})^{-1}$}\label{subsec5.3}
In the present subsection we will prove Theorem \ref{Unifcont} whose statement is given below.

Let\begin{equation}\label{modcont}
\sigma(t)=\left\{\begin{array}{ccc}
                   |\log t|^{-\frac{1}{8\delta}} &\mbox{for}& 0<t<\frac{1}{e}, \\
                   t-\frac{1}{e}+1& \mbox{for}&t\geq\frac{1}{e},
                 \end{array}
\right.
\end{equation}
where $\delta\in (0,1)$ is as  in Proposition \ref{QEUC}. The function $\sigma$ is strictly increasing, concave, and $\lim_{t\rightarrow 0}\sigma (t)=0$.   We have
\begin{theo}\label{Unifcont}
For every $\underline{L}^1, \underline{L}^2\in {\bold K}$ the following inequality holds true
\begin{equation}\label{stimarozza}
\|\underline{L}^1-\underline{L}^2\|_{\infty}\leq C_{\ast}\sigma^N(\|F(\underline{L}^1)-F(\underline{L}^2)\|_*)
\end{equation}
where $\sigma^N(\cdot)$  is the composition of the function $\sigma(\cdot)$ defined in (\ref{modcont}) with itself $N$ times and $C_{\ast}$ is a constant depending on $A$, $L$, $\alpha_0$, $\beta_0$, $N$ only.
\end{theo}
\begin{rem}
Observe that Theorem  \ref{Unifcont} provides  the injectivity of $F_{|\mathbf{K}}$ and an estimate of the modulus of continuity of $(F_{|\mathbf{K}})^{-1}$.
\end{rem}
In order to prove Theorem \ref{Unifcont} we need to prove first some preliminary results.
Let $j\in\{1,\ldots,N\}$ be such that
\[
\|\CC_{{\underline L}^1}-\CC_{{\underline L}^2}\|_{L^{\infty}(D_j)}=\|\CC_{{\underline L}^1}-\CC_{{\underline L}^2}\|_{L^{\infty}(\Omega_0)}
\]
and let $D_{j_1},\ldots,D_{j_M}$ be a chain of domains connecting $D_1$ to $D_j$. For the sake of brevity set $D_k=D_{j_k}$. Consider $\cal K$, $K_0$ and ${\cal K}_h$ defined as in Section 3. Let ${\cal W}_k=\textrm{Int}(\cup_{j=0}^k \overline{D}_j)$, ${\cal
U}_k=\Omega_0\backslash {\cal W}_k$, for $k=1,\dots, M-1$. The tensors $\CC_{\underline{L}^1}$ and $\CC_{\underline{L}^2}$ are extended as in (\ref{isotr1}) in all of
$\Omega_0$. To simplify the notation  we will set $\CC:=\CC_{{\underline L}^1}$
 and $\bar\CC:=\CC_{\underline{L}^2}$.
 Finally let $\tilde{{\cal K}}_k={\cal K}_h\cap {\cal W}_k$ and for $y,z\in \tilde{{\cal K}}_k$ define the matrix-valued function
\[
{\cal S}_k(y,z):=\int_{{\cal U}_k}(\CC-\bar\CC)(\cdot)\nabhat G(\cdot,y):\nabhat \bar G(\cdot,z) ,
\]
whose entries are given by
\[
{\cal S}^{(p,q)}_k(y,z):=\int_{{\cal U}_k}(\CC-\bar\CC)(\cdot)\nabhat G^{(p)}(\cdot,y):\nabhat \bar G^{(q)}(\cdot,z)\quad
p,q=1,2,3
\]
and where $G^{(p)}(\cdot,y)$ and $\bar G^{(q)}(\cdot,z) $ denote  respectively the $p$-th column and the $q$-th columnn of the singular solutions of Proposition
 \ref{Green} corresponding to the tensors $\CC$ and $\bar\CC$ respectively.
From (\ref{green1}) we have that
\[
|{\cal S}_k^{(p,q)}(y,z)|\leq  C(d(y)d(z))^{-1/2},  y,z\in \tilde{\cal K}_k,
\]
where  the constant $C$ depends on the a priori parameters only and $d(y)=d(y,{\cal U}_k)$, $d(z)=d(z,{\cal U}_k)$.
For any fixed $q=1,2,3$ let us denote by ${\cal S}^{(\cdot,q)}_k(\cdot,z)$ the vector valued function whose elements are ${\cal S}^{(p,q)}_k(y,z)$, $p=1,2,3$; analogously we define ${\cal S}^{(p,\cdot)}_k(y,\cdot)$, for any fixed $p=1,2,3$.

First we prove
\begin{prop}\label{sol}
For all $y,z\in \tilde{\cal K}_k$ we have that ${\cal S}_k^{(\cdot,q)}(\cdot, z)$ and ${\cal S}_k^{(p,\cdot)}(y, \cdot)$ belong to  $H^1_{loc}(\tilde{\cal K}_k)$ and for any
fixed $q\in\{1,2,3\}$
\begin{equation}\label{soly}
\mbox{div}(\CC  \nabhat {\cal S}_k^{(\cdot,q)}(\cdot,z))= 0\quad\textrm{in }\tilde{\cal K}_k,
\end{equation}
and for any fixed $p\in\{1,2,3\}$
\begin{equation}\label{solz}
\mbox{div}(\bar\CC  \nabhat{\cal S}_k^{(p,\cdot)}(\cdot, z))=0\quad\textrm{in }\tilde{\cal K}_k.
\end{equation}
\end{prop}
\begin{proof}
For seek of simplicity, in the proof we will omit the index $k$.
Let us fix $q\in \{1,2,3\}$ and let us first show that the vector valued function $S^{(\cdot,q)}(\cdot,z)\in H^1_{loc}(\tilde{\cal K})$  for fixed $z\in
\tilde{\cal K}$.  Let $\phi\in C^{\infty}_0(B_{r}(y_0))$ where $y_0\in \tilde{\cal K}$ and $B_{r}(y_0)\subset \tilde{\cal K}$. Consider, for fixed $p,
q\in\{1,2,3\}$
\begin{eqnarray*}
&&\int_{\tilde{\cal K}} {\cal S}^{(p,q)}(y,z)\partial_j \phi(y) dy=
\\&&\hskip 50pt=\int_{\tilde{\cal K}}\left[\int_{\cal U}(\CC-\bar\CC)(x)\nabhat_x G^{(p)}(x,y):\nabhat_x \bar
G^{(q)}(x,z) \partial_j\phi(y)dx\right]  dy.
\end{eqnarray*}
Observe now that by  (\ref{green1}) and by the fact that $B_{r}(y_0)\subset\tilde{\cal K}$ we have
\[
\int_{B_{r}(y_0)}\!\int_{\cal U}\!|\nabla_x G^{(p)}(x,y)|^2 dx dy<+\infty \textrm{,  }\quad \int_{B_{r}(y_0)}\!\int_{\cal U}\!|\nabla_x \bar G^{(q)}(x,z)|^2 dx dy<+\infty.
\]
Hence, by Schwarz inequality, we have that, for fixed $z\in\tilde{\cal K}$,
\[
 (\CC-\bar\CC)(x)\nabhat_x G^{(p)}(x,y):\nabhat_x \bar G^{(q)}(x,z) \partial_j\phi(y)\in L^1({\cal U}\times B_{r}(y_0)),
\]
so that we can interchange the order of integration and get
\begin{eqnarray*}
&&\int_{\tilde{\cal K}} {\cal S}^{(p,q)}(y,z)\partial_j \phi(y) dy=
\\&&\hskip 20pt=\int_{\cal U}\left[(\CC-\bar\CC)(x)\nabhat_x\int_{B_{r}(y_0)} G^{(p)}(x,y)\partial_j\phi(y)
dy:\nabhat_x \bar G^{(q)}(x,z) \right] dx
\end{eqnarray*}
and using the symmetry of $G$  almost everywhere in $\cal U$, (\ref{symgreen}), we get
\begin{eqnarray*}
\int_{B_{r}(y_0)}G^{(p)}(x,y)\partial_j\phi(y) dy&=&\int_{B_{r}(y_0)}G^{(p)}(y,x)\partial_j\phi(y) dy\\&=&-\int_{B_{r}(y_0)}\partial_{y_j} G^{(p)}(y,x)\phi(y) dy,
\end{eqnarray*}
so that
\begin{eqnarray*}
&&\int_{\tilde{\cal K}}{\cal S}^{(p,q)}(y,z)\partial_j \phi(y) dy=
\\&&\hskip 10pt=-\int_{\cal U}\left[(\CC-\bar\CC)(x)\nabhat_x\int_{B_{r}(y_0)} \partial_{y_j} G^{(p)}(y,x)\phi(y)
dy:\nabhat_x \bar G^{(q)}(x,z) \right] dx.
\end{eqnarray*}
Now recalling that  $G(y,x)=w(y,x)+\Gamma(y,x)$, by the properties of $\Gamma$  and by the boundary value problem satisfied by $\partial_{x_j} w(y,x)$ for any $j=1,2,3$  it is straightforward to see that $\nabla_y\partial_{x_j} w(y,x)\in L^2({\cal U}\times B_{r}(y_0))$ and hence
\[
\partial_{y_j}{\cal S}^{(p,q)}(y,z)=\int_{\cal U}(\CC-\bar\CC)(x) \nabhat_x\partial_{y_j} G^{(p)}(y,x) :\nabhat_x \bar G^{(q)}(x,z) dx.
\]

Now, arguing as in the first part of the proof and considering now a vector-valued test function $\Phi\in C^{\infty}_0(B_{r}(y_0))$ and by (\ref{symgreen}) we
have
\begin{eqnarray*}
&&\int_{B_{r}(y_0)}\!\!\!\CC(y)\nabhat_y{\cal S}^{(\cdot,q)}(y,z):\nabhat_y \Phi(y)dy=\\
&&=\int_{\cal U}(\CC-\bar\CC)(x)\nabhat_x\left[\int_{B_{r}(y_0)}\!\!\!\CC(y)\nabhat_y G(y,x)
:\nabhat_y\Phi(y) dy\right]:\nabhat_x \bar G^{(q)}(x,z)dx
\end{eqnarray*}
and since
\[
\int_{B_{r}(y_0)}\CC(y)\nabhat_yG^{(p)}(y,x):\nabhat_y \Phi(y)dy=0, \textrm{ a.e. in }\cal U
\]
for all $p=1,2,3$ we finally have
\[
\int_{B_{r}(y_0)}\CC(y)\nabhat_y{\cal S}^{(\cdot,q)}(y,z):\nabhat_y \Phi(y) dy=0,
\]
for all $\Phi\in C^{\infty}_0(B_{r}(y_0))$ and since $y_0$ is arbitrary \eqref{soly} follows. Analogously we get \eqref{solz}.
\end{proof}

\begin{prop}\label{QEUCSk}
If for a positive $\varepsilon_0$ and for some $k\in \{1,\cdots, M-1\}$
\begin{equation}\label{QEUCSk1}
|{\cal S}_k(y,z)|\leq \varepsilon_0r_0^{-1} \textrm{ for every }(y,z)\in  K_0\times K_0
\end{equation}
then
\begin{equation}\label{estimSk}
|{\cal S}_k(y_r,z_{\bar r})|\leq Cr_0^{-1}\left(\frac{r_0}{r}\right)^{5/2}\left(\frac{r_0}{\bar r}\right)^2\left(\frac{\varepsilon_0}{C_1+\varepsilon_0}\right)^{({\bar\theta}^{\bar nk})^2\tau_r\tau_{\bar r}}
\end{equation}
where $y_r=P_{k+1}+r n_{k+1}, z_{ \bar r}=P_{k+1}+\bar r n_{k+1}$, $P_{k+1}\in \Sigma_{k+1}$,  $\bar r, r\in (0,r_0/C)$, $\tau_r=\bar\theta\left(\frac{r}{r_0}\right)^{\delta}$, $\tau_{\bar r}=\bar\theta\left(\frac{\bar r}{r_0}\right)^{\delta}$ and $\bar n, C, C_1,\delta, \bar\theta\in (0,1)$ depend on $A,L, \alpha_0,\beta_0$ only.
\end{prop}
\begin{proof}
Fix $z\in K_0$ and consider the function $v(y):={\cal S}_k^{(\cdot,q)}(y,z)$, for fixed $q$. By Proposition \ref{sol} we know that $v$ is solution to
\[
\mbox{div}(\CC  \nabhat v(\cdot))= 0\quad\textrm{ in }\tilde{\cal K}_k.
\]
Moreover, from Proposition \ref{Green}, we get
\[
|v(y)|\leq C_1 r_0^{-1} , \quad y\in {\cal\tilde K}_k,
\]
where $C_1$ depends on $A,L, \alpha_0,\beta_0$ only
and from (\ref{QEUCSk1}) we have
\[
|v(y)|\leq  r_0^{-1}\varepsilon_0,\quad y\in  K_0.
\]
Then, applying Proposition \ref{QEUC} for $\varepsilon_1=r_0^{-1}\varepsilon_0$ and $E_1=C_1 r_0^{-1}$,  we have
\[
|v(y_r)|=|{\cal S}^{(\cdot,q)}_k(y_r,z)|\leq Cr_0^{-1}\left(\frac{r_0}{r}\right)^{2}\left(\frac{\varepsilon_0}{C_1+\varepsilon_0}\right)^{{\tilde\theta}^{\bar nk}\tau_r}.
\]
Now let us consider, for fixed $p$,
\[
\bar v(z):={\cal S}^{(p,\cdot)}_k(y_r,z)
\]
which is solution to
\[
\mbox{div}(\CC  \nabhat \bar v(\cdot))= 0\quad\textrm{in }\tilde{\cal K}_k
\]
and which satisfies
\[
|\bar v(z)|\leq r_0^{-1}C\left(\frac{r_0}{r}\right)^{2}
\left(\frac{\varepsilon_0}{C_1+\varepsilon_0}\right)^{{\tilde\theta}^{\bar nk}\tau_r}, \quad z\in K_0.
\]
By Proposition \ref{Green} we have
\[
|\bar v(z)|\leq Cr_0^{-1}\left(\frac{r}{r_0}\right)^{-1/2}\left(\frac{d(z)}{r_0}\right)^{-1/2}, \quad z\in {\cal\tilde K}_k,
\]
where $C$ depends on $A,L, \alpha_0,\beta_0$ only.
Hence, applying again Proposition \ref{QEUC} to $\bar v(z)$, we get
\[
|\bar v(z_{\bar r})|\leq  Cr_0^{-1}\left(\frac{r_0}{r}\right)^{5/2}\left(\frac{r_0}{\bar r}\right)^{2}\left(\frac{\varepsilon_0}{C_1+\varepsilon_0}\right)^{({\tilde\theta}^{\bar nk})^2\tau_r\tau_{\bar r}}
\]
which proves (\ref{estimSk}).
\end{proof}
\textit{Proof of  Theorem \ref{Unifcont}}\\
Observe first that $\|F(\underline{L}^1)-F(\underline{ L}^2)\|_*=\|\Lambda_{\CC}-\Lambda_{\bar\CC}\|_*$.
Denote by
\[
\varepsilon:=\|F(\underline{L}^1)-F(\underline{L}^2)\|_*.
\]
Then from identity (\ref{AlessId}), we derive that for every $y,z\in  K_0$ and for
$|l|=|m|=1$.
\begin{equation}\label{small}
\left|\int_{\om_0}(\CC-\bar\CC)(x)\nabhat G(x,y)\,l:\nabhat\bar G(x,z)\,m\, dx\right| \leq Cr_0^{-1}\varepsilon,
\end{equation}
where $C$ depends on $\alpha_0,\beta_0,A,L$ only. \\
Let
\[
\delta_k:=\max_{0\leq j\leq k}\{|\lambda_j-\bar\lambda_j|,|\mu_j-\bar\mu_j|\},
\]
where $k\in\{0,1,\cdots,M\}$. Recalling that by construction
\[
\CC|_{D_0}=\bar\CC|_{D_0}
\]
we have that $\delta_0=0$.
In order to obtain (\ref{stimarozza}) we use a recursive argument. More precisely, we prove that for a suitable increasing sequence $\{\omega_k(\varepsilon)\}_{0\leq k \leq M}$ satisfying $\varepsilon\leq\omega_k(\varepsilon)$ for every $k=0,\dots, M$ we have

\[
\delta_{k}\leq \omega_{k}(\varepsilon)\Longrightarrow \delta_{k+1}\leq \omega_{k+1}(\varepsilon),\mbox{ for every }k=0,\dots, M-1.
\]
Without loss of generality we can choose $\omega_{0}(\varepsilon)=\varepsilon$.
Suppose now that for some $k\in\{1,\cdots,M-1\}$ we have
\begin{equation}\label{delta}
\delta_k\leq \omega_k(\varepsilon).
\end{equation}
Consider
\[
{\cal S}_k(y,z):=\int_{{\cal U}_k}(\CC-\bar\CC)(\cdot)\nabhat G(\cdot,y):\nabhat \bar G(\cdot,z)
\]
and fix $z\in  K_0$. From Proposition \ref{Green}  and from (\ref{small}) we get, for $y, z\in  K_0$
\[
|{\cal S}_k(y,z)|\leq
\frac{C}{r_0}(\varepsilon+\omega_k(\varepsilon)),
\]
where $C$ depends on $A,L, \alpha_0,\beta_0$ only.
By (\ref{estimSk}) and choosing $\bar r=cr$ with $c\in [1/4,1/2]$ we easily get
that there are constants $C_0$, $\delta\in(0,1)$ and $\theta_\ast$ depending only on $A,L, \alpha_0,\beta_0$ and, increasingly, on $M$, such that
for fixed $l,m\in\RR^3$ such that $|l|=|m|=1$,
\begin{equation}\label{small2}
|{\cal S}_k(y_r,z_{\bar r})\,m\cdot l|\leq C r_0^{-1}\left(\frac{r_0}{r}\right)^{9/2}\varsigma \left(\omega_k(\ep),\frac{r}{r_0}\right),
\end{equation}
where
\begin{equation}\label{zitaelisa}
\varsigma\left(t,s\right)=\left(\frac{t}{1+t}\right)^{\theta_\ast s^{2\delta}}
\end{equation}
Let us choose $l=m=e_3$ and split
\begin{equation}\label{small3}
{\cal S}_k(y_r,z_{\bar r})\,e_3\cdot e_3=I_1+I_2,
\end{equation}
where
\begin{equation}\label{I1}
I_1= \int_{B_{r_1}\cap D_{k+1}}\!\!\!\!\!\!(\CC-\bar\CC)(x)\nabhat G(x,y_r) \,e_3:\nabhat\bar G(x,z_{\bar r})\,e_3\,dx
\end{equation}
and
\[
I_2= \int_{{\cal U}_{k+1}\backslash (B_{r_1}\cap D_{k+1})}\!\!\!\!\!\!(\CC-\bar\CC)(x)\nabhat G(x,y_r) \,e_3:\nabhat\bar G(x,z_{\bar r})\,e_3\,dx
\]
and where $r_1=\frac{r_0}{4LC_L}$ for $C_L$ as in (\ref{pho_1}). Then, from Proposition \ref{Green}, we derive immediately that
\begin{equation}\label{small4}
|I_2|\leq \frac{C}{r_0}.
\end{equation}
By \eqref{delta}  we have that
\begin{equation}\label{ind1}
|\overline{\lambda}_k-\lambda_k|\leq \omega_{k}(\varepsilon), \quad|\overline{\mu}_k-\mu_k|\leq \omega_{k}(\varepsilon).
\end{equation}
and hence, using \eqref{strongconv},
\begin{equation}\label{ind2}
|\overline{\nu}_k-\nu_k|\leq C\omega_{k}(\varepsilon),
\end{equation}
where $C$ depends on $\alpha_0$ only.
Estimates (\ref{ind1}) and (\ref{ind2}) together with (\ref{energy}) and (\ref{green1}) give
\begin{equation}\label{est2}
|I_1|\geq \left|\int_{B_{r_1}\cap D_{k+1}}\!\!\!\!\!\!\!\!\!\!\!\!\!\!\!\!\!(\CC_b^{k+1}-\bar\CC_b^{k+1})(\cdot)\nabhat \Gamma_{k+1}(\cdot,y_r)\, e_3:\nabhat\bar \Gamma_{k+1}(\cdot,z_{\bar r})\,e_3
\right|-\frac{C}{\sqrt{r_0 r}}-\frac{C\omega_k(\varepsilon)}{r},
\end{equation}
where $\Gamma_{k+1}$ and $\bar\Gamma_{k+1}$ are the biphase fundamental solutions introduced in Section \ref{sec4}
where the elastic phase correspond to elastic tensors $\CC_b^{k+1}$ and $\bar\CC_b^{k+1}$ given by
\begin{eqnarray*}
\CC_b^{k+1}&=&\CC_k\chi_{\RR^3_+}+\CC_{k+1}\chi_{\RR^3_-}\\
{\bar\CC}_b^{k+1}&=&{\CC}_k\chi_{\RR^3_+}+{\bar\CC}_{k+1}\chi_{\RR^3_-}\nonumber
\end{eqnarray*}
up to a rigid transformation that maps $\Sigma_{k+1}$ into $x_3=0$. Furthermore by (\ref{small2}), (\ref{small3}) and (\ref{small4}) we obtain
\begin{equation}\label{I1elisa}
|I_1|\leq Cr_0^{-1}\left(\left(\frac{r_0}{r}\right)^{9/2}\varsigma \left(\omega_k(\ep),\frac{r}{r_0}\right)+1\right),
\end{equation}
where $C$ depends on $A,L, \alpha_0,\beta_0$ only. Hence, by (\ref{I1elisa}) and \eqref{est2} and by performing the change of variables $x=rx^\prime$ in the integral at the right hand side of (\ref{I1}), we have
\begin{equation}\label{small6}
\left|\int_{B^-_{\rho_0/r}}\!\!\!\!\!\!\!\!
(\CC_b^{k+1}\!-\bar\CC_b^{k+1})(x^\prime)\nabhat \Gamma_{k+1}(x^\prime,e_3)e_3:\nabhat\bar \Gamma_{k+1}(x^\prime,ce_3)e_3\, dx^\prime\right|\leq \delta_0\!\left(\frac{r}{r_0}\right),
\end{equation}
where
\[\delta_0\left(\frac{r}{r_0}\right)=C\left[\left(\frac{r_0}{r}\right)^{7/2}\varsigma \left(\omega_k(\ep),\frac{r}{r_0}\right)+\sqrt{\frac{r}{r_0}}\,\right].\]
Since we have
\[
\left|\int_{\RR^3_-\backslash B^-_{\rho_0/r}}\!\!\!\!\!\!\!\!\!
(\CC_b^{k+1}-\bar\CC_b^{k+1})(x^\prime)\nabhat \Gamma_{k+1}(x^\prime,e_3)e_3:\nabhat\bar \Gamma_{k+1}(x^\prime,ce_3)e_3\, dx^\prime\right|\leq  C\frac{r}{r_0},
\]
where $C$ depends on $A,L, \alpha_0,\beta_0$ only, by \eqref{small6} and by applying Proposition \ref{trick} we obtain
\begin{equation}\label{small8}
\left|\left(\Gamma_{k+1}(e_3, ce_3)-\bar \Gamma_{k+1}(e_3,ce_3)\right) e_3\cdot e_3\right|\leq \delta_0\left(\frac{r}{r_0}\right)+C\frac{r}{r_0}.
\end{equation}

For seek of simplicity in what follows we will omit the indices $k$ and $k+1$
and write $\mu=\mu_k$, $\mu^\prime=\mu_{k+1}$ and, in a similar way, define
  $\lambda$, $\lambda^\prime$, $\nu$, $\nu^\prime$. We will bar corresponding Lam\'e coefficients for $\bar\CC$.

Using the explicit form (\ref{rongelisa}) of the biphase fundamental solution we have
\begin{eqnarray*}
&&\Gamma_{k+1}(e_3,ce_3)e_3\cdot e_3-\bar\Gamma_{k+1}(e_3,ce_3)e_3\cdot e_3=\frac{1}{4\pi(1-c)}\left(\frac{1}{\mu}-\frac{1}{\bar\mu}\right)+\\
&&\mbox{ } +\frac{1}{16\pi(1+c)}\left[\frac{\alpha[(3-4\nu)^2-\gamma+3-4\nu]}{\mu(1-\nu)}
-\frac{\bar\alpha[(3-4\bar\nu)^2-\bar\gamma+3-4\bar\nu]}{\bar\mu(1-\bar\nu)}\right]+\\
&&\mbox{ } +\frac{c}{4\pi(c+1)^3}\left[\frac{\alpha}{\mu(1-\nu)}-
\frac{\bar\alpha}{\bar\mu(1-\bar\nu)}\right],
\end{eqnarray*}
where
\begin{eqnarray*}
\alpha=F_1(\mu,\mu^\prime,\nu),&&\bar\alpha=F_1(\bar\mu,\bar\mu^\prime,\bar\nu) ,\\
\gamma=F_2(\mu,\mu^\prime,\nu,\nu^\prime),&&
\bar\gamma=F_2(\bar\mu,\bar\mu^\prime,\bar\nu,\bar\nu^\prime)\end{eqnarray*}
and $F_1$ and $F_2$ have been defined in (\ref{alpha}) and (\ref{gamma}).

From \eqref{small8}, \eqref{ind1} and \eqref{ind2} we obtain, by elementary calculation, for every $c\in [1/4,1/2]$
\begin{equation}\label{small13}
\left|p(c)\right|\leq C\left(\delta_0\left(\frac{r}{r_0}\right)+\frac{r}{r_0}+\omega_k(\varepsilon)\right)
\end{equation}
where
\[p(c):=[4(\alpha-\bar\alpha)(3-4\nu)(1-\nu)-(\alpha\gamma-\bar\alpha\bar\gamma)](1+c)^2+4c(\alpha-\bar\alpha) \]
and $C$ depends on $\alpha_0$ only.
Now, if $\omega_k(\varepsilon)<1/e$ then we choose $r=r_{\varepsilon}$ where
\[
r_{\varepsilon}=\frac{r_0}{C}|\log\omega_k(\varepsilon)|^{-\frac{1}{4\delta}},
\]
 where $C$ depends on $\alpha_0,\beta_0,A,L$ and $\delta\in (0,1)$ is as in Proposition \ref{QEUCSk} and by \eqref{small13} we get
 \begin{equation}\label{small15}
\left|p(c)\right|\leq
C|\log\omega_k(\varepsilon)|^{-\frac{1}{8\delta}},
\end{equation}
for every $c\in [1/4,1/2]$, where $C$ depends on $\alpha_0,\beta_0,A,L$ and $\delta\in (0,1)$.

Otherwise, if $\omega_k(\varepsilon)\geq 1/e$, since $p(c)$ is bounded, we can trivially write
 \begin{equation}\label{small16}
 \left|p(c)\right|\leq
Ce\omega_k(\varepsilon).
\end{equation}
Estimates (\ref{small15}) and (\ref{small16}) yields
\begin{equation}\label{small15ter}
\left|p(c)\right|\leq C\sigma(\omega_k(\varepsilon)),
\end{equation}
where $C$ depends on $\alpha_0,\beta_0,A,L$ only.

From  (\ref{small15ter}) we easily get
\begin{equation}\label{small17}
|\alpha-\bar\alpha|\leq C\sigma(\omega_k(\varepsilon)) \quad\textrm{,   }    |\alpha\gamma-\bar\alpha\bar\gamma|\leq C\sigma(\omega_k(\varepsilon)).
\end{equation}
where $C$  depends on $A$, $L$, $\alpha_0$ and $\beta_0$.

Now, by simple but tedious calculations, from \eqref{small17} and \eqref{boundPoisson} we derive
\[
|\nu^\prime-\bar\nu^\prime|\leq C \sigma(\omega_k(\varepsilon))\textrm{,  } |\lambda^\prime-\bar\lambda^\prime|\leq C \sigma(\omega_k(\varepsilon))\textrm{,  } |\mu^\prime-\bar\mu^\prime|\leq C \sigma(\omega_k\varepsilon))
\]
Hence we have
\[
\delta_{k+1}\leq \omega_{k+1}(\varepsilon):=C \sigma(\omega_k(\varepsilon)).
\]
Finally, by iteration and recalling that  $\omega_0(\varepsilon)=\varepsilon$ we get
\eqref{stimarozza}.
\qed

\subsection{Injectivity of $F^\prime(\underline{L})$ and estimate from below of $F^\prime_{|\mathbf{K}}$} \label{rifinita}

\begin{prop}\label{Stimabasso1}
Let $F: \mathcal{A}\rightarrow \mathcal{L}( H^{1/2}_{co}(\Sigma), H^{-1/2}_{co}(\Sigma))$ be the map introduced in definition \ref{forwardmap}. Let us define
\[
q_{0}:=\min \left \{ \left \Vert F^\prime\left( \underline{L}\right) [
\underline{H}]\right \Vert_* |\underline{L}\in \mathbf{K},\underline{H}\in \mathbb{R}
^{2N},\left \Vert \underline{H}\right \Vert _{\infty }=1\right \}
\]
we have
\begin{equation}\label{stimas_0}
\left(\sigma^N\right)^{-1}\left(1/C_{\star}\right)\leq q_{0},
\end{equation}
where $\sigma(\cdot)$ is defined by \eqref{modcont} and $C_{\star}$, $C_{\star}>1$, depends on $A$, $L$, $\alpha_0$, $\beta_0$ and $N$ only.
\end{prop}

\begin{proof}
By the definition of $q_0$ we have that there exist $\underline{L}_0 \in \mathbf{K}$ and $\underline{H}_0=(h_{0,1},\dots,h_{0,N}, k_{0,1},\ldots,k_{0,N})$, $\left \Vert \underline{H}_0\right \Vert _{\infty }=1$, such that

\[
\left\Vert F^\prime\left( \underline{L}_0\right) [\underline{H}_0] \right\Vert_*=q_0.
\]
Therefore, by \eqref{differenziale} and Proposition \ref{deriv} we have
\begin{equation}\label{1s_0}
\left\vert\int_{\om}
\HH_0 \nabhat G(\cdot,y)l:\nabhat G(\cdot,z)m\right\vert\leq \frac{C}{r_0}\,q_0 \quad\mbox{for every } y,z\in \mathcal{K}_0,
\end{equation}
for every $l,m$ unit vectors of $\mathbb{R}^3$, where $C$ depends on $\alpha_0,\beta_0$ $L$ and $A$ only, $\HH_0=\CC_{\underline{H}_0}$, $G(\cdot,y)$ denotes the singular solution defined in Section \ref{subsec4.3}.

\noindent From now on vector $(0,h_{0,1},\dots,h_{0,N},0, k_{0,1},\ldots,k_{0,N})$ will still be denoted by $\underline{H}_0$.

Let us fix $j\in\{1,\ldots,N\}$ and let $D_{j_1},\ldots,D_{j_M}$ be a chain of domains connecting $D_1$ to $D_j$, where
\[\max\{|h_{0,j}|,|k_{0,j}|\}=\left \Vert \underline{H}_0\right \Vert _{\infty }=1\]
 For the sake of brevity set $D_i=D_{j_i}$,
$i=1,...M$ and order domains and entries in $ \underline{H}_0$ accordingly. For every $i\in\{0,\ldots,M-1\}$ denote by $\Gamma_{i+1}(\cdot,\cdot)$ the  biphase fundamental  solution introduced in Section \ref{sec4} where the elastic phases correspond to the elastic tensor
given by
\[
\CC_b^{(i+1)}(x) = \CC_i\chi_{\RR^3_+}+\CC_{i+1}\chi _{\RR^3_-}(x).
\]
Here $\lambda_{0}=0$ and $\mu_0=1$.
Now, for any $i\in \{0,\dots, M-1\}$ let us denote by
\[
\eta_i:=\max \limits_{0\leq p\leq i}\left \{ \left \vert h_{0,p}\right \vert
,\left \vert k_{0,p}\right \vert \right \}.
\]
In order to obtain \eqref{stimas_0}, we use a recursive argument. More precisely, we prove that for a suitable increasing sequence $\{\omega_i(q_0)\}_{0\leq i \leq M}$ satisfying $q_0\leq\omega_i(q_0)$ for every $i=0,\dots,M$ we have

\[
\eta_{i}\leq \omega_{i}(q_0)\Longrightarrow \eta_{i+1}\leq \omega_{i+1}(q_0),
\mbox{ for every }i=0,\ldots,M-1,\]
from which, taking into account that we can choose $\omega_{0}(q_0)=q_0$, we will obtain  \eqref{stimas_0}.

For any $i\in \{0,\dots, M-1\}$, let ${\cal W}_i=\textrm{Int}(\cup_{j=0}^i \overline{D}_j)$, ${\cal U}_i=\Omega\setminus{\cal W}_i$, $\tilde{{\cal K}}={\cal K}_h\cap {\cal W}_i$  and, for $y,z\in \tilde{{\cal K}}$, let ${ {\cal T}}_i(y,z)=\left \{ {\cal T}^{(p,q)}_i(y,z)\right \}_{1 \leq p,q\leq 3} $ be the matrix valued function whose elements are given by
\[
{\cal T}^{(p,q)}_i(y,z):=\int_{{\cal U}_i}\HH_0(\cdot)\nabhat G^{(p)}(\cdot,y):\nabhat G^{(q)}(\cdot,z),
\quad p,q=1,2,3.
\]

Moreover, for any fixed $q=1,2,3$, let us denote by ${\cal T}^{(\cdot,q)}_i(\cdot,z)$ the vector valued function whose elements are ${\cal T}^{(p,q)}_i(\cdot,z)$, $p=1,2,3$; analogously we define ${\cal T}^{(p,\cdot)}_i(y,\cdot)$, for any fixed $p=1,2,3$.

Let us fix $l,m$ unit vectors of $\mathbb{R}^3$. By \eqref{1s_0} we have, for every $y,z\in K_0$,

\begin{eqnarray*}
\left\vert{\cal T}_{i}(y,z)\,l\cdot m\right\vert &\leq&\left\vert\int_{\om}
\HH_0 \nabhat G(\cdot,y)\,l:\nabhat G(\cdot,z)\,m\right\vert
\\ &+&\left\vert\int_{\om\cap {\cal W}_{i}}
\HH_0 \nabhat G(\cdot,y)\,l:\nabhat G(\cdot,z)\,m\right\vert\leq
\frac{C}{r_0}\,q_0+\frac{C}{r_0}\,\eta_{i}\nonumber\\
&\leq&\frac{C}{r_0}\left(q_0+\omega_{i}(q_0)\right).
\end{eqnarray*}
where $C$ depends on $\alpha_0$, $\beta_0$, $L$ and $A$ only.

Arguing similarly to the proof of Proposition \ref{QEUCSk}, we have that there exist $C_0$, $C_1$  such that for every $r\in(0,r_0/C_0)$ the following inequality holds true (recall $r_1=\frac{\rho_1}{4L}$ where $\rho_1$ is defined in \eqref{pho_1})

\begin{equation}\label{estSk}
|{\cal T}_{i}(y_r,z_{\bar r})|\leq C r_0^{-1}\left(\frac{r_0}{r}\right)^{9/2}\varsigma(\omega_i(q_0),r/r_0),
\end{equation}
where $\varsigma$ is defined as in (\ref{zitaelisa}).

Now we have trivially
\begin{eqnarray}\label{ssplit}
\int_{D_{i+1}\cap B_{r_1}}\!\!\!\!\!\!\!\!\!
\HH_0 \nabhat G(\cdot,y_r)\,l:\nabhat G(\cdot,z_{\bar r})\,m&=&\mathcal{T}_{i}(y_r,z_{\bar r})\,l\cdot m\\
&&-\int_{{\cal U}_{i}\setminus B_{r_1}}\!\!\!\!\!\HH_0 \nabhat G(\cdot,y_r)\,l:\nabhat G(\cdot,z_{\bar r})\,m.\nonumber
\end{eqnarray}
On the other side we have
\[
\left\vert\int_{{\cal U}_{i}\setminus B_{r_1}}\!\!\!\!\HH_0 \nabhat G(\cdot,y_r)\,l:\nabhat G(\cdot,z_{\bar r})\,m\right\vert\leq \frac{C}{r_0},
\]
where $C$ depend on $A$, $L$, $\alpha_0$, $\beta_0$ and $M$ only.

By the above inequality, \eqref{ssplit} and \eqref{estSk} we have

\[
\left\vert\int_{D_{i+1}\cap B_{r_1}}\HH_0 \nabhat G(\cdot,y_r)\,l:\nabhat G(\cdot,z_{\bar r})\,m\right\vert\leq \frac{C}{r_0}\left(1+\left(\frac{r_0}{r}\right)^{9/2}\varsigma(\omega_i(q_0),r/r_0)\right)
\]
where $C$ depend on $A$, $L$, $\alpha_0$, $\beta_0$ and $M$ only.

\noindent
Denote by
$\widetilde{\HH}$ the tensor  given by
\[
\widetilde{\HH}(x) =
 \left(h_{0,i+1}I_3\otimes I_3 +2k_{0,i+1}{\II}_{sym}\right)\chi _{\mathbb{R}
_{-}^{3}}(x).
 \]
From Proposition \ref{Green} we have, for every $c\in\left(1/4,1/2\right)$ and choosing $l=m=e_3$,

\begin{eqnarray}\label{Sgreen3}
\left \vert {\int_{B_{r_1}^-}
\widetilde{\HH} \nabhat \Gamma_{i+1}\left(x,re_3\right)e_3:\nabhat \Gamma_{i+1}\left(x,cre_3\right)e_3 dx}\right \vert\leq\nonumber\\
\hskip2cm\leq
\frac{C}{r_0}\left(\left(\frac{r_0}{r}\right)^{1/2}\!\!\!\!\!+\left(\frac{r_0}{r}\right)^{9/2}\varsigma(\omega_i(q_0),r/r_0)\right)
\end{eqnarray}
where  $C$ depend on $A$, $L$, $\alpha_0$, $\beta_0$ and $M$ only.

Now, performing the change of variables $x=r\xi$ in the integral on the right hand side of \eqref{Sgreen3} we obtain
\begin{eqnarray*}
\left \vert {\int_{B_{r_1/r}^-}\!\!\!
\widetilde{\HH} \nabhat \Gamma_{i+1}\left(\xi, e_3\right)e_3:\nabhat \Gamma_{i+1}\left(\xi, ce_3\right)e_3\, d\xi}\right \vert\leq \\ \hskip2cm\leq C\left[\left(\frac{r}{r_0}\right)^{1/2}+\left(\frac{r_0}{r}\right)^{7/2}\varsigma\left(\omega_i(q_0),\frac{r}{r_0}\right)\right].
\end{eqnarray*}
Therefore, for every $\varrho\in (0,1/C_1)$, we have

\begin{equation}\label{Shalfspace}
\left\vert\int_{\mathbb{R}^3_{-}}
\widetilde{\HH}\nabhat \Gamma_{i+1}\left(\cdot, e_3\right)e_3:\nabhat \Gamma_{i+1}\left(\cdot, ce_3\right)e_3\right\vert \leq C\left(\varrho^{1/2}+\varrho^{-7/2}\varsigma (\omega_i(q_0),\varrho)\right),
\end{equation}
where $C$ depend on $A, L, \alpha_0, \beta_0$ and $M$ only.
Now, if $\omega_{i}(q_0)<e^{-1}/2$  then we choose $\varrho=\frac{1}{2C_1}\left\vert\log\omega_{i}(q_0)\right\vert^{-\frac{1}{4\delta}}$, otherwise if $\omega_{i}(q_0)\geq e^{-1}/2$ then we estimate from above the right hand side of \eqref{Shalfspace} trivially. Hence, by Proposition \ref{trick2} we have

\begin{equation}\label{SVpsitrick}
 \left\vert\left(\frac{d}{dt}\Gamma_{\CC_b^{(i+1)}+t\widetilde{\HH}}(e_3,ce_3)\,e_3\cdot e_3\right)_{|t=0}\right\vert\leq C\sigma(\omega_{i}(q_0)),
\end{equation}
where $\sigma$ is defined by \eqref{modcont} and $C$ depends on $A$, $L$, $\alpha_0$, $\beta_0$ and $M$ only.

By explicit calculation from (\ref{rongelisa}), denoting by
\begin{eqnarray*}
  \lambda(t)=\lambda_{i+1}+th_{0,i+1}, && \mu(t)=\mu_{i+1}+tk_{0,i+1}, \\
  \nu(t)=\frac{\lambda(t)}{2(\lambda(t)+\mu(t))},&&
\end{eqnarray*}
and by
\[
\alpha(t)=F_1(\mu_i,\mu(t), \nu_i),\quad \gamma(t)=F_2(\mu_i,\mu(t), \nu_i,\nu(t)),
\]
for $F_1$ and $F_2$ as in \eqref{alpha} and \eqref{gamma}, we get
\begin{eqnarray*}
&&\frac{1}{16\pi\mu_i(1-\nu_i)}\left(\frac{d}{dt}\Gamma_{\CC_{i+1}+t\widetilde{\HH}}(e_3,ce_3)m\cdot l\right)_{|t=0}=\\
&&\hskip1cm=\frac{1}{(1+c)^3}\left\{\left[4(1-\nu_{i})(3-4\nu_{i})\alpha^\prime(0
)+(\alpha\gamma)^\prime(0)\right](1+c)^2+4c\alpha^\prime(0)\right\}.
\end{eqnarray*}
Therefore, from (\ref{SVpsitrick}) and \eqref{boundPoisson} we find easily
\begin{equation}\label{Vsyst1}
\left\{\begin{array}{rcl}
\left\vert \alpha^\prime(0)\right\vert\leq C\sigma(\omega_{i}(q_0)),\\
\left\vert (\alpha\gamma)^\prime(0) \right\vert\leq C\sigma(\omega_{i}(q_0)).
\end{array}\right.
\end{equation}
The first condition of \eqref{Vsyst1} gives

\[
\left\vert\frac{4(1-\nu_{i})\mu_{i}}{(\mu_{i}+(3-4\nu_{i})\mu_{i+1})^2}k_{0,i+1}\right\vert\leq C\sigma(\omega_{i}(q_0))
\]
hence, by recalling (\ref{boundPoisson}) and \eqref{strongconv}, we have
\begin{equation}\label{k1}
\left\vert k_{0,i+1}\right\vert\leq C\sigma(\omega_{i}(q_0)).
\end{equation}
Taking into account \eqref{k1}, the second equation of \eqref{Vsyst1} implies
\[
\left\vert\frac{8(1-\nu_{i})\mu_{i}\mu_{i+1}^2}{2(\lambda_{i+1}+\mu_{i+1})^2(\mu_{i+1}+(3-4\nu_{i+1})\mu_{i})^2}h_{0,i+1}\right\vert\leq C\sigma(\omega_{i}(q_0)),
\]
hence, again by \eqref{boundPoisson} and \eqref{strongconv},
\[
\left\vert h_{0,i+1}\right\vert\leq C\sigma(\omega_{i}(q_0)),
\]
where $C$ depends on $A$, $L$, $\alpha_0$, $\beta_0$ and $M$ only.

Therefore
\[
\eta_{i+1}\leq \omega_{i+1}(q_0):= C\sigma(\omega_{i}(q_0)),
\]
where $C$ depends on $A$, $L$, $\alpha_0$, $\beta_0$ and $M$ only.

Finally, by iteration we get
\[1=\eta_M\leq C\sigma^M(q_0)\leq C \sigma^N(q_0)\]
and the thesis follows.
\end{proof}
\begin{rem}\label{bomba}
Observe that the above proposition implies that the Frech\'et derivative $F^\prime(\underline{L})$ is injective for every $\underline{L}\in\mathbf{K}$ and hence point (v) of Proposition \ref{propBV} is completely proved.
Therefore we have
\[\|\underline{L}^1-\underline{L}^2\|_{\infty}\leq C\|F(\underline{L}^1)-F(\underline{L}^2)\|_* \quad \forall \underline{L}^1,\underline{L}^2\in {\bold K},\]
where $C=\max\{\frac{2R_1}{(\sigma_2)^{-1}(\delta_1)},\frac{2}{q_0}\}$, $M_1=\min\{\frac{\beta_0}{\sqrt{13}},\alpha_0\}$, $M_2=\frac{\sqrt{2N}}{\alpha_0}$, $\sigma_2(\cdot)=C_{\ast}\sigma^N(\cdot)$, $q_0=\left(\sigma^N\right)^{-1}\left(1/C_{\star}\right)$,  $\delta_1=\frac{1}{2}\min\{\delta_0, M_2\}$ and $\delta_0=\frac{q_0}{2C_{F^\prime}}$ and we recall that $C_{\ast}$ is the constant that occurs in Theorem \ref{Unifcont}, $C_{F^\prime}$ is the Lipschitz constant of $F^\prime$ introduced in Proposition \ref{deriv}, and $C_{\ast}$ has been introduced in  Proposition \ref{Stimabasso1}.
\end{rem}

\section{Appendix }
\label{appendix}
For the convenience of the reader, we recall here some quantitative estimates of unique continuation. Although such estimates have been proved in the general case where the elasticity tensor is of class $C^{1,1}$, here we give the statements in the special case we are interested in, namely  we assume that
\begin{equation}\label{isotrConst}
\CC=\lambda I_3\otimes I_3 +2\mu{\II}_{sym},
\end{equation}
where $\lambda$ and $\mu$ are real numbers satisfying (\ref{strconvlame}).

The following theorem is an immediate consequence of \cite[Theorem 5.1]{Al-M}
and standard estimate of smallness propagation \cite[proof of Theorem 1.10]{A-R-R-V}

\begin{theo}[Three sphere inequality]
\label{lem:3sphere} Let $u$ be a solution to the Lam\'{e} system
\[
\mbox{div}\left({\CC}\widehat{\nabla}u\right)=0\quad\mbox{in}\quad B_{\bar r},
\]
for some positive number $\bar r$. Then, for every $r_1, r_2, r_3$, such that
$0<r_1\leq r_2<r_3\leq\bar r$, we have
\begin{equation}
  \label{eq:3sph}
   \int_{B_{r_{2}}}| u|^{2} \leq C
   \left(  \int_{B_{r_{1}}}|u|^{2}
   \right)^{\theta_0}\left(  \int_{B_{r_{3}}}| u|^{2}
   \right)
   ^{1-\theta_0}\!\!\!\!,
\end{equation}
where $C$ and $\theta_0$, $0<\theta_0<1$, only depend on
$\alpha_{0}$, $\beta_{0}$, $\frac{r_{2}}{r_{3}}$ and,
increasingly, on $\frac{r_{1}}{r_{3}}$ .
\end{theo}

The following theorem has been proved in \cite{M-R1}.
\begin{theo}[Stability estimate for the Cauchy problem]
  \label{theo:Cauchy-generale}

Let $\CC$ be as in (\ref{isotrConst}). Let $u$ be the
solution to the Cauchy problem
\[
\left\{\begin{array}{rcl}
             \mbox{div}\left(\CC\widehat{\nabla}u\right)&=&0\mbox{ in } B^+_{ r_0/3}, \\
             u(x^\prime,0) & =& h(x^\prime)\mbox{ on } B^\prime_{r_0/3},\\
             \frac{\partial u}{\partial x_3}(x^\prime,0)&=&g(x^\prime) \mbox{ on } B^\prime_{r_0/3},
           \end{array}\right.
\]
where $h \in
H^{\frac{1}{2}} (B^\prime_{r_0/3})$ and $g \in H^{- \frac{1}{2}}(B^\prime_{r_0/3})$.  We have

\begin{multline*}
  \|u\|_{L^\infty (B^+_{ r_0/6})}
  +
  r_0 \|\nabla u\|_{L^\infty (B^+_{ r_0/6})}
  \leq
  \\
  \leq
  C
  \|u\|_{H^1 (B^+_{ r_0/6})}
  ^{1-\theta_1}
  \left (
  \|h\|_{L^2 (B^\prime_{r_0/3})}
  +
  \| g \|_{H^{- \frac{1}{2}}(B^\prime_{r_0/3})}
  \right )^{\theta_1},
\end{multline*}
where $C$ and $\theta_1$, $0<\theta_1<1$, only depend on $\alpha_0$,
$\beta_0$.
\end{theo}

\bigskip

In order to state the following result of smallness propagation in a cone (Proposition \ref{cono}) we introduce some notation.
Given $z\in \mathbb{R}^{3}$, $\zeta \in
\mathbb{R}^{3}$, $\left \vert \zeta \right \vert =1$, $\gamma \in \left( 0,
\frac{\pi }{2}\right)$, we denote by
\[
C\left( z,\zeta ,\gamma \right) =\left \{ x\in \mathbb{R}^{3}|\text{ \ }
\frac{(x-z)\cdot \zeta }{\left \vert x-z\right \vert }>\cos \gamma \right \}
\]
the open cone having vertex $z$, axis in the direction $\zeta $ and width $
2\gamma $ and, for any $\rho >0$, we denote by
\[
C_{\rho }(\gamma )=C\left( 0,-e_{3},\gamma \right) \cap Q_{\rho ,H_{\gamma
} },
\]
where $H_{\gamma }=\frac{1}{\tan \gamma }$.

Let $\gamma _{1},\gamma _{2},\gamma _{3}\in \left( 0,\frac{\pi }{2}\right) $
be such that $\gamma _{1}<\gamma _{2}<\gamma _{3}$.

Denote by
\[
t_{0}=\frac{H_{\gamma _{3}}\rho }{1+\sin \gamma _{3}}\text{, }
\]
\[
\chi =\frac{1-\sin \gamma _{2}}{1-\sin \gamma _{1}}\text{,}
\]
and, for any $t\in \left( 0,t_{0}\right] $, denote by
\[
s_{k}=\chi ^{k-1}t\text{, \ }w_{k}=-s_{k}e_{3}\text{, \ }k\in \mathbb{N}
\text{,}\]
\begin{equation}\label{asterisco}
r_{3}^{(k)}=s_{k}\sin \gamma _{3}\text{, }r_{2}^{(k)}=s_{k}\sin \gamma _{2}
\text{, }r_{1}^{(k)}=s_{k}\sin \gamma _{1}\text{, \ }k\in \mathbb{N}\text{.}
\end{equation}
Notice that we have
\begin{equation}
\label{palle}
B_{r_{1}^{(k+1)}}\left( w_{k+1}\right) \subset B_{r_{2}^{(k)}}\left(
w_{k}\right) \subset B_{r_{3}^{(k)}}\left( w_{k}\right) \subset C_{\rho
}(\gamma _{3})\text{, \ for every }k\in \mathbb{N}\text{.}
\end{equation}

Let $r$ be a given number such that $r\in \left( 0,\chi t_{0}\right] $. Let $%
k_{0}$ be the smallest integer number such that $\chi ^{k-1}\leq \frac{r}{%
t_{0}}$ and let
\[
t:=\chi ^{-\left( k_{0}-1\right) }r\text{.}
\]
Notice that
\[
\chi t_{0}\leq t\leq t_{0}
\]
and
\begin{equation}
\label{notice}
w_{k_{0}}=-re_{3}\text{, }r_{1}^{(k_{0})}=r\sin \gamma _{1}\text{.}
\end{equation}

\bigskip

\begin{prop}\label{cono}
Let $\CC$ be as in \eqref{isotrConst} and let $u$ be a solution to the Lam\'{e} system
\[
\mbox{div}\left({\CC}\widehat{\nabla}u\right)=0\quad\mbox{ in }C_{\rho }(\gamma _{3}).
\]

Assume that
\begin{equation}
\label{E}
\int_{B_{t\sin \gamma _{1}}\left( w_{1}\right) }|u|^{2}\leq \varepsilon ^{2}%
\text{, \  \ }\int_{C_{\rho }(\gamma _{3})}|u|^{2}\leq E^{2}\text{,}
\end{equation}
where $\varepsilon ,E$ are given positive numbers such that $\varepsilon \leq E$.
Then we have
\begin{equation}
\label{erre}
|u(-re_{3})|\leq \frac{C}{r^{3/2}}\varepsilon ^{\eta _{r}}E^{1-\eta _{r}}
\text{,}
\end{equation}
where
\[
\eta _{r}=\overline{\theta }\left( \frac{r}{t}\right) ^{\frac{\left \vert
\log \overline{\theta }\right \vert }{\left \vert \log \chi \right \vert }}%
\text{,}
\]
$C$ and $\overline{\theta }$, $\overline{\theta }\in (0,1)$ depend on $\alpha _{0}$, $\beta _{0}$, $\gamma _{1}$, $\gamma _{2}$ and $\gamma _{3}$ only.
\end{prop}

\textit{Proof of Proposition \ref{cono}}\\
Let $r_{i}^{\left( k\right) }$ for $i=1,2,3$ be as in (\ref{asterisco}).
By Theorem \ref{lem:3sphere} we have, for every $k\in \mathbb{N}$,
\begin{equation}
\label{trepalle_k}
\int_{B_{r_{2}^{\left( k\right) }}\left( w_{k}\right) }|u|^{2}\leq C\left(
\int_{B_{r_{1}^{\left( k\right) }}\left( w_{k}\right) }|u|^{2}\right) ^{%
\overline{\theta }}\left( \int_{B_{r_{3}^{\left( k\right) }}\left(
w_{k}\right) }|u|^{2}\right) ^{1-\overline{\theta }},
\end{equation}
where $C$ and $\overline{\theta }$, $\overline{\theta }\in (0,1)$, depend on
$\alpha _{0}$, $\beta _{0}$, $\gamma _{1}$, $\gamma _{2}$ and $\gamma _{3}$ only.

Denote by
\[
\sigma _{k}:=E^{-2}\int_{B_{r_{1}^{\left( k\right) }}\left( w_{k}\right)
}|u|^{2}\text{.}
\]
Since, by (\ref{palle}),  $B_{r_{1}^{\left( k\right) }}\subseteq B_{r_{2}^{\left( k-1\right) }}$, we have
\begin{equation}
\label{1s}
\sigma _{k}\leq E^{-2}\int_{B_{r_{2}^{\left( k-1\right) }}\left(
w_{k-1}\right) }\!\!\!\!\!\!\!|u|^{2},\quad\text{ for every }k\geq 2\text{.}
\end{equation}
By the second inequality in (\ref{E}) and by (\ref{trepalle_k}) and (\ref{1s}) we get
\[
\sigma _{k}\leq C\sigma _{k-1}^{\overline{\theta }}\text{, for every }k\geq 2.
\]
Iterating the last inequality and taking into account the first inequality in (\ref{E})
we get
\[
\sigma _{k}\leq C^{\frac{1}{1-\overline{\theta }}}\left( \frac{\varepsilon
^{2}}{E^{2}}\right) ^{\overline{\theta }^{k-1}}\text{, for every }k\geq 2%
\text{.}
\]
Now, we choose $k=k_{0}$ in the above inequality and notice that
\[
\overline{\theta }^{k_{0}-1}\geq \eta _{r}\text{,}
\]
hence
\[
\int_{B_{r_{1}^{\left( k_{0}\right) }}\left( w_{k_{0}}\right) }|u|^{2}\leq
C\varepsilon ^{2\eta _{r}}E^{2\left( 1-\eta _{r}\right) }\text{,}
\]
where $C$ depends on $\alpha _{0}$, $\beta _{0}$, $\gamma _{1}$, $\gamma _{2}$ and $\gamma _{3}$ only.

Finally, by (\ref{notice}) and by the estimate (\ref{stimaLinfinito}),
\[
\left \Vert u\right \Vert _{L^{\infty }\left( B_{\frac{r\sin \gamma _{1}}{2}}\left(
-re_{3}\right) \right) }^{2}\leq \frac{C}{\left( r\sin \gamma _{1}\right) ^{3}%
}\int_{B_{r\sin \gamma _{1}}\left( re_{3}\right) }|u|^{2}\text{,}
\]
where $C$ depends on $\alpha _{0}$ and $\beta _{0}$ only and (\ref{erre})
follows.  $\Box$

\bigskip
We finally end this appendix by proving our main result on quantitative estimate of unique continuation for solutions of Lam\'{e} system with piecewise constant coefficients. In this proof the elasticity tensor $\CC$ is of the form \eqref{isotr1}.

\bigskip

\textit{Proof of Proposition \ref{QEUC}}\\
Denote by $C_2=\max \left \{6, 4LC_{L},4C_{1},\frac{2r_0}{h_0} \right\}$ and by $\rho_2=\frac{r_0}{C_2}$, where $C_{L}$ and $C_1$ are defined in \eqref{pho_1} and in Proposition \ref{Analytic} respectively and $h_0$ is defined in (\ref{h_0}). Notice that $C_2$ does not depend on $r_0$ and that $\rho_2\leq\frac{\rho_1}{16}$. It is not restrictive to assume $n_1=e_3$. Let us denote $x_0=P_{1}+n_1\frac{3}{16}\rho_1$.
We have by \eqref{palla in K_0}
\[
  K_{0}\supset B_{\rho _{2}}^{\prime }(P_{1})\times \left( \frac{\rho _{1}}{8
},\frac{\rho _{1}}{4}\right)\supset B_{\rho _{2}}(x_0).
\]
Moreover, by Proposition \ref{Analytic} we have that the function $v_{|B_{\rho _{2}}^{\prime }(P_{1})\times \left(0,\frac{\rho _{1}}{4}\right)}$ can be extended
analytically to a function $v_0$ on $B_{\rho _{2}}^{\prime }(P_{1})\times \left(-\frac{r_{0}}{4C_{1}},\frac{\rho _{1}}{4}\right)$
and
\begin{equation}
  \label{Bound}
\|v_0\|_{L^\infty\left(B_{\rho _{2}}^{\prime}(P_{1})\times \left( -\frac{\rho _{1}}{4},\frac{r_{0}}{4C_{1}
}\right)\right)}\leq C (E_1+\ep_1),
\end{equation}
where $C$ depends on $A$, $L$, $\alpha_0$, $\beta_0$ and $N$ only.

Let us construct a chain of spheres of radius $\rho_2/4$ such that the first is $B_{\rho_2/4}(x_0)$, all the spheres are externally tangent and the last one is
centered at $P_1+\frac{\rho_2}{2}n_1$. We choose such a chain so that the spheres of radius $\rho_2$ concentric with those of the chain are contained in
$B_{\rho _{2}}^{\prime}(P_{1})\times \left(-\frac{r_{0}}{4C_{1}},\frac{\rho _{1}}{4}\right)$. The number of spheres of the chain is certainly smaller than a
constant $m_1$ depending on $L, \alpha_0$ and $\beta_0$ only.

By an iterated application of three sphere inequality \eqref{eq:3sph} with $r_1=\frac{\rho_2}{4}, r_2=\frac{3\rho_2}{4}, r_3=\rho_2$ and by \eqref{Bound} we have
\begin{equation}
  \label{error1}
\|v_0\|_{L^2\left(B_{3\rho_2/4}(P_1+\frac{\rho_2}{2}n_1)\right)}\leq C\ep_1^{\theta_0^{m_1}}(E_1+\ep_1)^{1-\theta_0^{m_1}},
\end{equation}
where $\theta_0$, $0<\theta_0<1$, depends on $\alpha_0$ and $\beta_0$ only and $C$ depends on $A, L, \alpha_0$ and $\beta_0$ only.
Since $B_{3\rho_2/4}(P_1+\frac{\rho_2}{2}n_1)\supset B_{\rho_2}(P_1)$, by \eqref{error1} we have trivially
\begin{equation}
  \label{error2}
\|v_0\|_{L^2\left(B_{\rho_2/4}(P_1)\right)}\leq C\ep_1^{\theta_0^{m_1}}(E_1+\ep_1)^{1-\theta_0^{m_1}}.
\end{equation}
By the above inequality and by Caccioppoli inequality, \cite[pag.20]{BBFM}, we have
\begin{equation}
  \label{error3}
\|\nabla v_0\|_{L^2\left(B_{\rho_2/8}(P_1)\right)}\leq \frac{C}{r_0}\ep_1^{\theta_0^{m_1}}(E_1+\ep_1)^{1-\theta_0^{m_1}},
\end{equation}
where $C$ depends on $A, L, \alpha_0$ and $\beta_0$ only.

By \eqref{error2} and \eqref{error3} we get the following trace inequality
\begin{equation}
  \label{error4}
\|v_0\|_{L^2 (B^\prime_{\rho_2/16}(P_1))}
  + r_0\| (\CC\nabhat v_0)n_1 \|_{H^{- \frac{1}{2}}(B^\prime_{\rho_2/16}(P_1))}\leq C\ep_1^{\theta_0^{m_1}}(E_1+\ep_1)^{1-\theta_0^{m_1}}.
\end{equation}
Now, let us recall the following transmission conditions
\begin{equation}
  \label{transmission}
  {v}_{|D_0}={v}_{|D_1} \mbox{ on } \Sigma_1 \mbox{ , } \left(\CC\nabhat v\right)_{|D_0}n_1= \left(\CC\nabhat v\right)_{|D_1}n_1 \mbox{ on } \Sigma_1.
\end{equation}
 Let us denote by $B^-_{\rho_2/32}(P_1)=D_1\cap B_{\rho_2/32}(P_1)$. By Theorem \ref{theo:Cauchy-generale}, \eqref{transmission}, \eqref{error4}, \eqref{limitaz}
 and the Caccioppoli inequality  we have
 \begin{equation}
 \label{errorv_1}
 \|v_{|D_1}\|_{L^\infty (B^-_{\rho_2/32}(P_1))}
  +
  r_0 \|\nabla v_{|D_1}\|_{L^\infty (B^-_{\rho_2/32}(P_1))}
  \leq
  C\ep_1^{\theta_1\theta_0^{m_1}}(E_1+\ep_1)^{1-\theta_1\theta_0^{m_1}},
\end{equation}
where  $\theta_0\in(0,1)$ and $\theta_1\in(0,1)$ depend on $\alpha_0$ and $\beta_0$ only and $C$ depends on $A, L, \alpha_0$ and $\beta_0$ only.

Now, we prove by induction what follows:
given $m_2=\frac{3A (64)^{3(N-2)}C_2^3}{4\pi}$ and $\rho_k=(64)^{-k+2}$,
there exist a constant $C$  depending on $A$, $L$, $\alpha_0$, $\beta_0$ and (increasingly) on $M$, such that for
every $k\in\{1,2, ..., M-1\}$ the following inequality holds true
\begin{multline}
 \label{errorv_k}
 \|v_{|D_k}\|_{L^\infty (B^-_{\bar\rho_{k+1}/32}(P_k))}
  +
  r_0 \|\nabla v_{|D_k}\|_{L^\infty (B^-_{\bar\rho_{k+1}/32}(P_k))}
  \leq\\
  \leq C\ep_1^{\theta_1^k\theta_0^{m_1+(k-1)m_2}}(E_1+\ep_1)^{1-\theta_1^k\theta_0^{m_1+(k-1)m_2}}.
\end{multline}

If $k=1$ then \eqref{errorv_k} is proved in \eqref{errorv_1}. Now, assume that \eqref{errorv_k} holds true for $k$, $1\leq k\leq M-2$. Let us denote by $v_{k}$
the analytic extension of $v_{|\widetilde{\mathcal{K}}_{h_0/2}\cap D_{k}}$ to $\left(\widetilde{\mathcal{K}}_{h_0/2}\cap D_{k+1}\right)\cup(\Xi^{C_1}_{k+1}\cap\bar D_{k})$ where
$h_0$, $\widetilde{\mathcal{K}}_{h_0/2}$ and $\Xi^{C_1}_{k+1}$ are defined by \eqref{h_0}, \eqref{Corridoio} and \eqref{Xi} respectively. Let us denote
$x_k=P_k-\frac{\rho_k}{32}n_k$. By the induction hypothesis we have trivially

\begin{equation}
 \label{errorv_2}
 \|v_k\|_{L^2 (B_{\rho_{k+1}/64}(x_k))}
    \leq
  C\ep_1^{\theta_1^k\theta_0^{m_1+(k-1)m_2}}(E_1+\ep_1)^{1-\theta_1^k\theta_0^{m_1+(k-1)m_2}},
\end{equation}
where $C$ depends on $A, L, \alpha_0, \beta_0$ and (increasingly) $k$ only.

Let us construct a chain of spheres of radius $\rho_{k+1}/4\cdot64$ such that the first is $B_{\rho_{k+1}/4\cdot64}(x_k)$, all the spheres are externally
tangent and the last one is centered at $P_{k+1}+\frac{\rho_{k+1}}{4\cdot64}n_{k+1}$. We choose such a chain so that the spheres of radius
$\rho_{k+1}/64$ concentric with those of the chain are contained in $\widetilde{\mathcal{K}}_{h_0/2}\cap D_{k}$. The number of spheres of the chain is
certainly smaller than a constant $m_2$.

By an iterated application of three sphere inequality \eqref{eq:3sph} with $r_1=\frac{\rho_{k+1}}{4}$, $r_2=\frac{3\rho_{k+1}}{4}$, $r_3=\rho_{k+1}$ and
by \eqref{errorv_2} we have
\begin{multline*}
  \|v_k\|_{L^2\left(B_{\rho_{k+1}/4}(P_{k+1})\right)}\leq \|v_k\|_{L^2\left(B_{3\rho_{k+1}/4}(P_{k+1}+\frac{\rho_{k+1}}{4}n_{k+1})\right)}\leq\\
  \leq
  C\ep_1^{\theta_1^k\theta_0^{m_1+km_2}}(E_1+\ep_1)^{1-\theta_1^k\theta_0^{m_1+km_2}},
\end{multline*}
where $C$ depends on $A, L, \alpha_0, \beta_0$ and (increasingly) $N$ only.

Now, proceeding in exactly the same way followed to prove \eqref{errorv_1} we have
\begin{multline*}
 \|v_{|D_{k+1}}\|_{L^\infty (B^-_{\rho_{k+1}/32}(P_{k+1}))}
  +
  r_0 \|\nabla v_{|D_{k+1}}\|_{L^\infty (B^-_{\rho_{k+1}/32}(P_{k+1}))}
  \leq\\
  \leq C\ep_1^{\theta_{1}^{k+1}\theta_0^{m_1+km_2}}(E_1+\ep_1)^{1-\theta_{1}^{k+1}\theta_0^{m_1+km_2}},
\end{multline*}
where $C$ depends on $A, L, \alpha_0, \beta_0$ and $N$ only.
Therefore, inequality \eqref{errorv_k} holds true for every $k\in\{1,2, ..., M-1\}$.

In particular by \eqref{errorv_k} we have
\begin{multline}
 \label{errorv_5}
 \|v_{|D_{M-1}}\|_{L^\infty (B^-_{\rho_{M}/32}(P_{M}))}
  +
  r_0 \|\nabla v_{|D_{M-1}}\|_{L^\infty (B^-_{\bar\rho_{M}/32}(P_M))}
  \leq\\
  \leq C\ep_1^{\theta_1^{M-1}\theta_0^{m_1+(M-2)m_2}}(E_1+\ep_1)^{1-\theta_1^{M-1}\theta_0^{m_1+(M-2)m_2}}.
\end{multline}
Now by a rigid transformation of coordinate under which we have $P_{M}=0$  and $
n _{M}=e_3$ we have $Q_{(M)}^{-}=B_{\rho _{1}/4L}^{\prime }\times \left(0,\frac{\rho _{1}}{4}\right)$.
In what follows, first we derive an error smallness estimate in a ball contained in the cylinder $B_{\rho _{1}/8L}^{\prime }\times \left(0,\frac{\rho
_{1}}{4}\right)$ and then we apply  Proposition \ref{cono}  in the cone $C_{\rho _{1}/8L }(\gamma _{3})$, where $\gamma _{3}=\arctan \frac{1}{2L}$. In order to
make precise such a step we adopt all the notation introduced before such a proposition. Therefore we denote $H_{\gamma _{3}}=2L$,
$
t_{0}=\frac{\rho _{1}}{8L}\frac{H_{\gamma _{3}}}{1+\sin \gamma _{3}}$. Since
$\sin \gamma _{3}=\frac{1}{\sqrt{4L^{2}+1}}$ we have
\[
t_{0}=\frac{\rho _{1}}{4}\frac{\sqrt{4L^{2}+1}}{1+\sqrt{4L^{2}+1}}.
\]
Let $\gamma _{1},\gamma _{2}$ be such that
\[
\sin \gamma _{1}=\frac{1}{4}\sin \gamma _{3}\text{, \ }\sin \gamma _{2}=%
\frac{3}{4}\sin \gamma _{3}
\]
and set
\[
\chi =\frac{1-\sin \gamma _{2}}{1-\sin \gamma _{1}}=\frac{4\sqrt{4L^{2}+1}-3%
}{4\sqrt{4L^{2}+1}-1}\text{.}
\]

Let $r\in \left( 0,\chi t_{0}\right) $ be fixed and
\[
t:=\chi ^{-(k_{0}-1)}r
\]
where $k_{0}$ is the smallest integer number such that $\chi ^{k-1}\leq
\frac{t_{0}}{r}$. Moreover let $\overline{P}=P_{M}+tn _{M}$.
Proceeding as before, by \eqref{errorv_5} we get
\[
 \|v\|_{L^2 (B_{\rho_{M}/4\cdot64}(\overline{P}))}
    \leq
  \ep_2:=C\ep_1^{\theta_1^{M-1}\theta_0^{m_1+(M-1)m_2}}(E_1+\ep_1)^{1-\theta_1^{M-1}\theta_0^{m_1+(M-1)m_2}},
\]
where $C$ depends on $A$, $L$, $\alpha_0$, $\beta_0$ and $M$ only.

\noindent Now denote by $\rho=\min \left \{\rho_{M}/4\cdot64,t\sin \gamma_1\right \}$. By applying Theorem \ref{lem:3sphere} with $r_1=\rho$, $r_2=t\sin
\gamma_1$ and $r_3=t\sin \gamma_3$ we get

\begin{equation*}
 \|v\|_{L^2 (B_{t\sin \gamma_1}(\overline{P}))}
    \leq
  \ep_2:=C\ep_2^{\theta_2}(E_1+\ep_2)^{1-\theta_2},
\end{equation*}
where $C$, $\theta_2$, $0<\theta_2<1$ depend on $L$, $\alpha_0$ and $\beta_0$ only.
Now we use Proposition \ref{cono} and we get
\begin{equation}
\label{erreV}
|v(\tilde{x})|\leq C\left( \frac{r_{0}}{r}\right) ^{2-\frac{\eta_{r}}{2}}\varepsilon_2 ^{\eta _{r}}(E_1+\ep_2)^{1-\eta _{r}}
\text{,}
\end{equation}
where $\tilde{x}=P_{M}+rn_{M}$, $C$ depends on $A$, $L$, $\alpha_0$, $\beta_0$ and (increasingly) $N$ only and
\begin{equation}
\label{etarV}
\eta _{r}=\theta_2\left( \frac{r}{t}\right) ^{\frac{\left \vert
\log \theta_2\right \vert }{\left \vert \log \chi \right \vert }}.
\end{equation}
Finally, denoting by $\widetilde{\theta}=\min \left \{\theta_1, \theta_2, \theta_3 \right \}$ and $\overline{m}=\max \left \{m_1+1, m_2+1 \right \}$ we get
\begin{equation*}
|v(\tilde{x})|\leq C\left( \frac{r_{0}}{r}\right) ^{2-\frac{\eta_{r}}{2}}\varepsilon_1 ^{\widetilde{\theta}^{\overline{m}M}\eta
_{r}}(E_1+\ep_1)^{1-\widetilde{\theta}^{\overline{m}M}\eta _{r}}
\end{equation*}
and by $t\geq\chi t_0$ the thesis follows. \qed


\bibliographystyle{plain}

\end{document}